\documentclass[reqno,11pt]{preprint}
\usepackage{preamble}
\usepackage[toc,page]{appendix}

\begin{document}

\title{Stationary stochastic Navier-Stokes on the plane at and above criticality}
\author{G. Cannizzaro and J. Kiedrowski}

\institute{University of Warwick, CV4 7AL, UK\\
	\email{\mbox{giuseppe.cannizzaro@warwick.ac.uk, j.kiedrowski@warwick.ac.uk}}}

\maketitle
\begin{abstract}
In the present paper, we study the fractional incompressible Stochastic Navier-Stokes equation on $\R^2$, 
formally defined as
\begin{equation}\label{e:snsabs}
\partial_t v = -\tfrac12 (-\Delta)^\theta v - \lambda v \cdot \grad v + \grad p -  \pgrad (-\Delta)^{\frac{\theta-1}{2}} \xi, \qquad \grad \cdot v = 0 \, ,
\end{equation}
where $\theta\in(0,1]$, $\xi$ is the space-time white noise on $\R_+\times\R^2$ and $\lambda$ is the coupling constant. 
For any value of $\theta$ the previous equation is ill-posed due to the singularity of the noise, and is 
critical for $\theta=1$ and 
supercritical for $\theta\in(0,1)$. For $\theta=1$, we prove that the weak coupling regime for the equation, 
i.e. regularisation at scale $N$ and coupling constant $\lambda=\hat\lambda/\sqrt{\log N}$, 
is meaningful in that the sequence $\{v^N\}_N$ of regularised solutions is tight and 
the nonlinearity does not vanish as $N\to\infty$. Instead, for $\theta\in(0,1)$ 
we show that the large scale behaviour of $v$ is trivial, as 
the nonlinearity vanishes and $v$ is simply converges to the solution of~\eqref{e:snsabs} 
with $\lambda=0$. 
%
%
%
%
% and tuning the coupling constant , we 
% Regularisation via a Fourier cut-off is imposed on the non-linear term and its removal is investigated. Using an analogous equation on $\T^2$ it is shown that the invariant measure on $\T^2$ induces an invariant measure on $\R^2$ via the use of a martingale problem and energy bounds. Under re-scaling $\lambda \sim (\log N )^{-\half}$ the non-linear term is shown to converge to a non-trivial limit as the cut-off is removed. In addition, it is shown that the solutions to regularised equations are tight in $C_T \mathcal{S}^{\prime} (\R^2,\R)$ with the same scaling. The super-critical regime is investigated, where it turns out the weak-coupling scaling is too strong, causing the non-linear term to vanish in the limit.
\end{abstract}

\setcounter{tocdepth}{1}       
%\tableofcontents

\section{Introduction}

The incompressible Navier-Stokes equation is a partial differential equation (PDE) describing the motion of 
an incompressible fluid subject to an external forcing. It is given by 
\begin{equ}[e:ns]
	\partial_t v = \tfrac12 \Delta v - \hat\lambda v \cdot \grad v + \grad p -  f\,,\qquad\nabla\cdot v=0\,,
\end{equ}
where $v=v(t,x)$ is the velocity of the fluid at $(t,x)\in\R_+\times \R^d$, 
$\hat\lambda\in\R$ is the {\it coupling constant} which tunes the strength of the nonlinearity,
$p$ is the pressure, $f$ the forcing and the second equation is the incompressibility condition. 
For $f$ a random noise (which will be the case throughout the paper), we will refer to 
the above as to the Stochastic Navier-Stokes (SNS) equation. 

The SNS equation has been studied under a variety of assumptions on $f$. Most of the literature focuses 
on the case of trace-class noise, for which existence, uniqueness of solutions and ergodicity 
were proved (see e.g.~\cite{FG, DaD, HM, FR, RZZ} and the references therein). 
The case of even rougher noises, e.g. space-time white noise and its derivatives, which is relevant 
in the description of motion of turbulent fluids~\cite{mikuleviciusStochasticNavierStokes2004}, 
was first considered in $d=2$ in~\cite{DaD2}, 
and later, thanks to the theory of Regularity Structures~\cite{Hai} and the paracontrolled calculus approach~\cite{GIP}, 
in dimension three~\cite{ZZ}. 

In the present work, we focus on dimension $d=2$ and consider the fractional stochastic Navier-Stokes equation 
driven by a conservative noise, which formally reads
\begin{equ}[e:sns]
	\partial_t v = -\tfrac12 (-\Delta)^\theta v - \hat\lambda v \cdot \grad v + \grad p -  \pgrad (-\Delta)^{\frac{\theta-1}{2}} \wn\,, \qquad \grad \cdot v = 0 \, .
\end{equ}
Here, $\theta$ is a strictly positive parameter, $(-\Delta)^\theta$ is the usual fractional laplacian 
(see~\eqref{e:FracLaplacian} below for the definition of its Laplace transform),  
$\pgrad  \eqdef (\partial_2 ,- \partial_1  )$ and $\wn$ is a space-time white noise on $\R_+\times\R^2$, 
i.e. a Gaussian process whose covariance is given by 
\begin{equ}[e:CovWN]
	\Ex [\wn(\phi) \wn(\psi) ] = \qvar{\phi , \psi }_{L^2 (\R_+ \times \R^2 ) }\,, \qquad \forall \,\phi, \psi \in L^2 (\R_+ \times \R^2)\, .
\end{equ}
The choice of the forcing $f=\pgrad (-\Delta)^{\frac{\theta-1}{2}} \wn$ in~\eqref{e:sns} ensures that, 
at least formally, the spatial white noise on $\R^2$, i.e. the Gaussian process 
whose covariance is that in~\eqref{e:CovWN} but 
with $\R^2$-valued square-integrable $\phi,\,\psi$, is invariant for the dynamics. 

A rigorous analysis of~\eqref{e:sns} has so far only been carried out for 
$\theta>1$, which in the language of~\cite[Ch. 8]{Hai}, corresponds to the so-called
{\it subcritical} regime - in~\cite{gubinelliRegularizationNoiseStochastic2013}, the authors 
proved existence of stationary solutions while uniqueness 
was established in~\cite{Gubinelli.TurraHyperviscous2020}. 
The goal of the present paper is instead to study the large-scale behaviour of the 
fractional SNS in the {\it critical} and {\it supercritical} cases, i.e. $\theta=1$ and $\theta\in(0,1)$ respectively,  
for which not only the classical stochastic calculus tools but also 
the pathwise theories of Regularity Structures~\cite{Hai} and paracontrolled calculus~\cite{GIP} 
are not applicable. 

To motivate our results, let us first consider the vorticity formulation of~\eqref{e:sns}.  
Setting $\omega\eqdef \grad^{\perp} \cdot v$, $\omega$ solves 
\begin{equ}[e:vsns]
	\partial_t \omega = -\tfrac12 (-\Delta)^\theta \omega - \hat\lambda\, (K\ast \w) \cdot \grad \omega -  (-\Delta)^{\frac{\theta+1}{2}} \wn\,, %\qquad \grad \cdot v = 0 \, .
\end{equ}
where $K$ is the Biot-Savart kernel on $\R^2$ given by 
%\giuseppeText{ to specify that we apply to functions which are supported away from the origin in Fourier}
\begin{equ}[e:biotsavart]
	K(x) \eqdef \frac{1}{ \iota} \int_{\R^2} \frac{y^{\perp}}{ |y|^2 } e^{-\iota y \cdot x}\dd y\, ,
\end{equ}
for $y^{\perp} \eqdef (y_2, -y_1)$. Note that the velocity $v$ can be fully recovered 
from the vorticity $\w$ via $v=K\ast \omega$, so that~\eqref{e:sns} and~\eqref{e:vsns} are indeed equivalent. 

Due to the roughness of the noise, 
as written~\eqref{e:vsns} is purely formal for any value of $\theta\in(0,1]$. 
Therefore, in order to work with a well-defined object, we first regularise the equation. 
Let $\rho^1$ be a smooth spatial mollifier, the superscript $1$ representing the scale of the regularisation,
and consider the regularised vorticity equation 
\begin{equs}[e:vsnsNTsup]
	\partial_t \w^{1} = -\tfrac{1}{2}\D{\theta} \w^{1} -  \hat\lambda \grad^{\perp} \cdot \rho^1 \ast \left( (K\ast (\rho^1 \ast w)) \cdot \grad (\rho^1 \ast w) \right)   + \D{\frac{1+\theta}{2}} \xi \, .
\end{equs}
Since we are interested in the large-scale behaviour of~\eqref{e:vsns}, we rescale 
$\w^1$ according to
\begin{equ}[e:rescaling]
	\w^{N}(t,x) \eqdef N^{2} \w^1 ( t{N^{2\theta}} , x{N}) \, ,
\end{equ}
so that $\wN$ solves
\begin{equs}[e:vsnsNTsupintro]
	\partial_t \w^N = -\tfrac{1}{2}\D{\theta} \w^N -  \hat\lambda N^{2\theta-2} \nnlin [\w^N]   + \D{\frac{1+\theta}{2}} \xi \, ,
\end{equs}
and the nonlinearity $\nnlin$ is defined according to 
\begin{equs}[e:SNSnonlin]
	\nnlin[\w] \eqdef \div \mol \ast \left( (K\ast (\mol \ast \w)) (\mol \ast \w) \right)\, .
\end{equs}
where $\mol(\cdot)\eqdef N^{2}\rho(N^2\cdot)$ and we used $\grad \cdot v =0$. 

Note that as an effect of the scaling~\eqref{e:rescaling}, the coupling constant $\hat\lambda$ gains an 
$N$-dependent factor which, for large $N$, is order $1$ for $\theta=1$, i.e. in the critical regime, 
while it vanishes polynomially for $\theta\in(0,1)$, which instead is the supercritical regime. 
The goal of the present paper is twofold. For $\theta\in(0,1)$, 
we will show that the nonlinearity {\it simply goes to $0$} 
and that the equation trivialises, in the sense that $\wN$ converges to the solution 
of the original fractional stochastic heat equation obtained by setting $\hat\lambda=0$ in~\eqref{e:vsnsNTsupintro}. 
At criticality, i.e. $\theta=1$, instead the situation is more subtle. 
Logarithmic corrections due to the nonlinear term are to be expected (see~\cite{WAG,LRY} 
and~\cite{cannizzaro2021stationary, cannizzaro2021sqrtlog} for other models in the same universality class) 
and need to be taken into account. In the present setting, we will do so by imposing that 
the coupling constant vanishes at a suitable logarithmic order (see~\eqref{e:CouplingConstant}). 
We will then show that this is 
indeed meaningful since on the one hand subsequential limits for $\wN$ exist 
and on the other the nonlinear term does not vanish but is uniformly (in $N$) of order $1$. 
 
Before delving into the details, let us state assumptions, scalings and results more precisely. 
To unify notations, for $N\in\N$ let $\wN$ be the solution of 
\begin{equ}[e:vsnsN]
	\partial_t \wN =-\tfrac{1}{2}\D{\theta} \w^N - \l \nnlin [\w^N]   + \D{\frac{1+\theta}{2}} \xi  \, ,\qquad \w(0,\cdot)=\w_0(\cdot)
\end{equ}
where $\w_0$ is the initial condition, the value of $\l$ depends on both $N$ and $\theta$ via
\begin{equ}[e:CouplingConstant]
\l\eqdef
\begin{cases}
\frac{\hat\lambda}{ \sqrt{\log N}}\,,& \text{for $\theta=1$}\\
\hat\lambda N^{2\theta-2}\,,& \text{for $\theta\in(0,1)$,}
\end{cases}
\end{equ}
$\nnlin$ is defined according to~\eqref{e:SNSnonlin} with $\mol$ satisfying 
the following assumption. 

\begin{assumption}\label{a:mol}
For all $N\in\N$, $\mol$ is a radially symmetric smooth function such that $\|\mol\|_{L^1 (\R^2)}=1$ 
and whose Fourier transform $\molf$ is compactly supported on $\{k :1/N < |k | < N\}$. Furthermore, there 
exists a constant $c_\rho>0$ such that
\begin{equs}[ass:rho]
  |\molf(k)| \geq c_{\rho}, \quad \forall k \in \{k :2/N < |k | < N/2 \}\, .
\end{equs}
We also define $\mol_y (\cdot) \eqdef \mol (\cdot -y)$. 
\end{assumption}

\begin{remark}\label{rem:a}
Assumption~\ref{a:mol} guarantees that for every $N$, $\nnlin[\w]$ is smooth even if $\w$ is 
merely a distribution. In particular, $\molf$ is assumed to be vanishing in a neighbourhood of $0$ 
in order to avoid problems arising from  
the singularity of the Fourier transform Biot-Savart kernel $K$~\eqref{e:biotsavart} on $\R^2$ at the origin.
\end{remark}

Our first result concerns existence, markovianity and stationarity of the regularised vorticity equation~\eqref{e:vsnsN} 
for fixed $N$ and any value of $\theta\in(0,1]$. 

\begin{theorem}\label{thm:introSNSstation}
Let $\theta\in(0,1]$. For any $T>0$ and $N\in\N$ fixed,~\eqref{e:vsnsN} 
admits a weak solution $\wN$ in $C([0,T],\cS'(\R^2))$, which is a strong Markov process and 
has as invariant measure the Gaussian field $\mu$ whose covariance is 
\begin{equ}[e:muintro]
\E[\mu(\phi)\mu(\psi)]\eqdef \qvar{\phi,\psi}_{\dot H^1(\R^2)}\,,\qquad \phi,\,\psi\in \dot H^1(\R^2)\,.
\end{equ}
(For the definition of the space $\dot H^s(\R^2)$, $s\in\R$, see~\eqref{e:Hsnorm} below.)
\end{theorem}

The previous theorem is shown in Section~\ref{s:invmeas} and Theorem~\ref{thm:InvR2}. 
As noted in~\cite[Remark 3.1-(2)]{Funaki.QuastelKPZ2015}, 
since we are working on the full space $\R^2$, invariance of the field 
$\mu$ in~\eqref{e:muintro} cannot be obtained via the classical Echeverria's 
criterion~\cite{echeverriaCriterionInvariantMeasures1982}. 
Instead, we suitably modify a strategy first appeared in~\cite{Funaki.QuastelKPZ2015} 
which consists of approximating $\wN$ with the sequence of periodic solutions 
to~\eqref{e:vsnsN} on a large torus of size $M$ for which invariance can be easily established.  

From now on we will only work with the stationary solution to~\eqref{e:vsnsN}. 
In the next theorem, we study the critical case of $\theta=1$ and prove that, 
if $\lambda_{N,1}$ is chosen according to~\eqref{e:CouplingConstant}, then the sequence 
$\{\wN\}_N$ is tight and that the subsequential limits of the non-linearity do not vanish. 

\begin{theorem}\label{thm:Main}
For $N\in\N$, let $\wN$ be the stationary solution of~\eqref{e:vsnsN} with $\theta=1$, $\lambda_{N,1}$ defined 
according to~\eqref{e:CouplingConstant} for $\hat\lambda>0$, $\mol$ satisfying Assumption~\ref{a:mol} 
and initial condition $\w_0=\mu$, for $\mu$ the Gaussian field 
with covariance as in~\eqref{e:muintro}. For $\phi\in\cS(\R^2)$ and $t\geq 0$,  
set 
\begin{equ}[e:IntNonlin]
\cB^N_t(\phi)\eqdef \lambda_{N,1}\int_0^t\nnlin_t[\w^N_s](\phi)\dd s\,.
\end{equ}
Then, for any $T>0$, the law of the couple $(\w^N, \cB^N)$ is tight in $C([0,T],\cS'(\R^2))$. 
Moreover, letting $(\w,\cB)$ be any limit point, we have that 
there exists a constant $C>1$ such that 
for all $\phi\in\cS(\R^2)$ and $\kappa>0$
\begin{equ}[e:UpperLowerB]
C^{-1}\frac{\hat\lambda^2}{\la^2}\| \phi \|_{\dot{H}^2 (\R^2)}^2\leq \int_0^{\infty} e^{-\la t} \Ex \Big[ \Big| \cB_t(\phi) \Big|^2 \Big] \dd t \leq C\frac{\hat\lambda^2}{\la^2} \| \phi \|_{\dot{H}^2 (\R^2)}^2\,.
\end{equ} 
\end{theorem}

The proof of the previous statement will occupy Sections~\ref{s:TightUp} and~\ref{s:lowerbound} and 
is based on techniques similar to those used in~\cite{Cannizzaro.etal2D2021}, which though 
works on a torus of fixed size instead of $\R^2$. 
For both tightness and the upper bound in~\eqref{e:UpperLowerB},
we exploit the so-called It\^o's trick introduced in~\cite{gubinelliRegularizationNoiseStochastic2013} 
(see Theorem~\ref{t:tightvsnsN} and Remark~\ref{rem:UB}), 
while the lower bound is achieved via a variational problem and suitable operator bounds 
on the generator of $\wN$ (see Proposition~\ref{p:lowerboundsns} and Lemma~\ref{l:OpBounds}). 

\begin{remark}
The previous theorem suggests that, under the scaling determined by 
the choice of the coupling constant in~\eqref{e:CouplingConstant} (the so-called weak coupling scaling), 
the equation~\eqref{e:vsnsN} is 
{\it diffusive} at large scales. We believe that,  following the approach developed in~\cite{CETweak} 
it should be possible to show that $\wN$ indeed converges in law and that the limit is 
a stochastic heat equation with renormalised $\hat\lambda$-dependent coefficients. 
\end{remark}

At last, we consider the supercritical regime $\theta\in(0,1)$. As previously anticipated, 
in this case the nonlinearity simply converges to $0$ so that $\wN$ 
trivialises. 

\begin{theorem}\label{thm:theta<1}
For $N\in\N$ and $\theta\in(0,1)$, 
let $\wN$ be the stationary solution of~\eqref{e:vsnsN} with $\l$ defined 
according to~\eqref{e:CouplingConstant} for $\hat\lambda>0$, $\mol$ satisfying Assumption~\ref{a:mol} 
and initial condition $\w_0=\mu$, for $\mu$ the Gaussian field 
with covariance as in~\eqref{e:muintro}. Then, the sequence $\{\wN\}_N$ converges as $N\to\infty$ 
to the unique solution of the fractional stochastic heat equation
\begin{equ}[e:fracSHE]
\partial_t \omega=-\tfrac12(-\Delta)^\theta \omega +(-\Delta)^{\frac{1+\theta}{2}}\xi\,,\qquad \w_0=\mu\,.
\end{equ}
\end{theorem}

The proof of the previous theorem is given in Theorem~\ref{t:tightvsnsN} and Section~\ref{s:triv}. 
We believe that a similar type of statement will hold true also for other fractional equations, 
e.g. the fractional Anisotropic KPZ equation with $\theta<1$ or the one dimensional fractional KPZ equation 
with $\theta<1/2$ (see~\cite{KiedrowskiSPDES2021} for the first and~\cite{gubinelliRegularizationNoiseStochastic2013, gubinelliInfinitesimalGeneratorStochastic2018} 
for the other). 

\section*{Acknowledgements}
The authors would like to thank Mario Maurelli and Marco Romito 
for helpful discussions. 
G. C. gratefully acknowledges financial support via the EPSRC grant EP/S012524/1. 
\subsection*{Notations and function spaces}

For $M \in \N$ let  $\T_M^2$ be the two dimensional torus of side length $2\pi M$ and 
$\Z_M^2 \eqdef (\Z_0/M)^2$ where $\Z_0 \eqdef \Z \backslash \{ 0 \}$. 
Denote by $\{ e_k\}_{k \in \Z_M^2}$ the usual Fourier basis, i.e. 
$e_k (x) \eqdef \frac{1}{2\pi} e^{ik \cdot x}$, 
and for $\phi \in L^2(\T_M^2)$ let the Fourier transform of $\phi$ be
\begin{equ}
	\F_M(\phi)(k)  = \hat\phi (k) = \phi_k \eqdef \int_{\T_M^2} \phi (x) e_{-k}(x) \dd x\, ,
\end{equ}
so that, in particular, for all $x \in \T^2_M$ we have
\begin{equ}[e:RiemSum]
	\phi (x) = \frac{1}{M^2} \sum_{k \in \Z_M^2} \hat\phi(k) e_k (x)\, .
\end{equ}
When the space considered is clear, the subscript M may be omitted. 
The previous definitions straightforwardly translate to $\R^2$ by replacing the integral over the torus 
to the full space, the Riemann-sum to an integral and taking $k\in\R^2$. 
%Similarly we equip the space $L^2 ( \R^2)$ with Fourier basis $\{ e^{\prime}_k \}_{k \in \R^2}$ defined as $e^{\prime}_k (x) \eqdef e^{ik \cdot x}$ and for $\phi \in L^2 (\R^2)$ define the Fourier transform by
%\begin{equ}
%	\F(\phi)(k) \eqdef \int_{\R^2} \phi (x) e^{\prime}_{-k}(x) \dd x\, ,
%\end{equ}
%so that for $x \in \R^2$ we have
%\begin{equ}
%	\phi(x) = \int_{\R^2} \phi (x) e^{\prime}_{k}(x) \dd x \, .
%\end{equ}

For $\theta\in \R$ and $T=\T^2_M$ or $\R^2$, 
we define the fractional Laplacian $(-\Delta)^{\theta}$ via its Fourier transform, i.e.
\begin{equ}[e:FracLaplacian]
	\mathcal{F} ( (-\Delta)^{\theta} u ) (z) = |z|^{2\theta} \mathcal{F}(u) (z)\, ,
\end{equ}
for $\phi \in L^2(T)$ and $z \in T$ for $\theta \geq 0$ and $z \in T \backslash \{ 0\}$ otherwise. 

We denote by $\mathcal{S} (\R^2)$, the classical space of Schwartz functions, i.e. infinitely differentiable functions 
whose derivatives of all orders decay at faster than any polynomial. 
%\begin{equ}
%	\cS (\R^2) \eqdef \{ \phi \in \mathcal{C}^{\infty} (\R^2) :\| \phi\|_{\alpha, \beta} = \sup_{x\in \R^2} \big| x^{\alpha}D^{\beta} \varphi (x) \big| < \infty \: \forall \alpha, \beta \in \N^2 \}\, ,
%\end{equ}
%where we used the multi-index notation $x^{\alpha} = x_1^{\alpha_1} x_2^{\alpha_2}, D^{\beta} = \partial_1^{\beta_1} \partial_2^{\beta_2}$. 
Similarly to~\cite[Section 7]{Gubinelli.TurraHyperviscous2020}, 
for $s \in \R$, we say $\phi \colon (\R^2)^n \rightarrow \R$ is in the \textit{homogeneous Sobolev space} 
$(\dot{H}^s (\R^2))^{\otimes n}$, understood as a tensor product of Hilbert spaces, 
if there exists a tempered distribution $\tilde{\phi} \in \mathcal{S}^{\prime} ((\R^2)^n)$ such that
\begin{equs}
	\qvar{\phi,\psi}_{(\dot{H}^s (\R^2))^{\otimes n}} = \qvar{\tilde\phi,\psi}_{(\dot{H}^s (\R^2))^{\otimes n}}\, \forall \psi \in \mathcal{S}((\R^2)^n)\,.
\end{equs}
and 
\begin{equ}[e:Hsnorm]
	\| \tilde\phi \|^2_{(\dot{H}^s (\R^2))^{\otimes n} } \eqdef \int_{(\R^2)^n} \left( \prod_{i=1}^{n} |k_i|^{2s} \right) |\hat{\tilde\phi} (k_{1:n})|^2 \dd k_{1:n}<\infty
\end{equ}
where we introduced the notation $k_{1:n}\eqdef (k_1,\dots,k_n)$. 
Clearly, for $s\geq 0$, $\tilde\phi$ can be taken to be $\phi$ itself. 
The same conventions apply to $\dot{H}^s (\T_M^2)$, 
but in the definition of the norm the integral is replaced by a Riemann-sum (as in~\eqref{e:RiemSum}). 

For $s=1$, which will play an important role in what follows, we point out that the norm 
on $(\dot H^1(\R^2))^{\otimes n}$ can be equivalently written as
\begin{equ}
\| \phi \|^2_{(\dot{H}^1 (\R^2))^{\otimes n} }\eqdef\int_{(\R^2)^n} |\nabla \phi(x_{1:n})|^2\dd x_{1:n}\,.
\end{equ}

%We say that a function $f$ on $\R^2$ is \textit{symmetric} if for any permutation $\sigma : \{ 1, \dots  , n \} \rightarrow \{1, \dots, n\}$ we have
%\begin{equ}
%	f(x_1,\dots, x_n) = f(x_{\sigma(1)}, \dots, x_{\sigma(n)})\, ,
%\end{equ}
%and denote the space of integrable symmetric functions by $L^2_{\sym} (\R^2)$. We denote the space of \textit{tempered distributions } $\cS'$ on $\R^2$ to be the dual of $\mathcal{S} (\R^2)$. 

\subsection{Preliminaries on Wiener space analysis}\label{sec:Mall}

Let $(\Omega, \mathcal{F},\mathbb{P})$ be a complete probability space and 
$H$ be a separable Hilbert space with scalar product $\qvar{\cdot,\cdot} $. 
A stochastic process $\mu$ is called \textit{isonormal Gaussian process} 
(see~\cite[Definition 1.1.1]{nualartMalliavinCalculusRelated2006})
if $\{ \mu(h) : h \in H \} $ is a family of centred jointly Gaussian random variables 
with correlation $\E ( \mu(h) \mu(g) ) = \qvar{h,g}$. 
Given an isonormal Gaussian process $\mu$ on $H$ and $n \in \N$, 
we define the \textit{n-th homogeneous Wiener chaos} $\wc_n$ 
as the closed linear subspace of $L^2 (\mu) = L^2 (\Omega)$ 
generated by the random variables $H_n ( \mu(h) )$, for $h \in H$ of norm $1$, 
where $H_n$ is the $n$-th Hermite polynomial. 
%defined recursively via
%\begin{equ}
%	H_0 (x) = 1, H_n(x) = \frac{(-1)^n}{n!}e^{x^2 /2 } \frac{d^n}{\dd x^n } \left( e^{-x^2 /2} \right)\, .
%\end{equ}
For $m \neq n$, $\wc_n$ and $\wc_m$ are orthogonal
and, by \cite[Theorem 1.1.1]{nualartMalliavinCalculusRelated2006}, $L^2 (\mu) = \oplus_n \wc_n$. 

The isonormal Gaussian process $\mu$ we will be mainly working with is such that $H=\dot H^1(T)$, 
$T$ being either the $2$-dimensional torus $\T^2_M$ or $\R^2$, and has covariance  
\begin{equ}[e:mu]
\E[\mu(\phi)\mu(\psi)]\eqdef \qvar{\phi,\psi}_{\dot H^1(T)}\,,\qquad \phi,\,\psi\in \dot H^1(T)\,.
\end{equ}
Thanks to the results in~\cite[Chapter 1]{nualartMalliavinCalculusRelated2006}, 
there exists an isomorphism $I$ between the Fock space $\fock \eqdef \oplus_{n \geq 0 } \fock_n$ 
and $L^2 (\mu)$, where $\fock_n$ is the closure of $(\dot H^1 (T))^{\otimes n}$ 
with respect to the norm in~\eqref{e:Hsnorm}\footnote{Equivalently, $\fock_n$ is 
the space of functions in $\dot H^1 (T^n) $ which are symmetric with respect to permutations of variables}. 
For $n\in\N$, the projection $I_n$ of the isomorphism above to $\wc_n$ is itself an
isomorphism between $\fock_n$ and $\wc_n$ and is given by 
\begin{equ}
I_n(\otimes^n h)\eqdef n! H_n(\mu(h))\,,\qquad \text{for all $h\in \dot H^1(T)$ such that $\|h\|_{\dot H^1(T)}=1$.}
\end{equ}
By~\cite[Theorem 1.1.2]{nualartMalliavinCalculusRelated2006}, 
for every $F \in L^2 (\mu)$ there exists unique sequence of symmetric functions 
$\{ f_n \}_{n \geq 0} \in \fock$ such that $F = \sum_{n=0}^\infty I_n (f_n)$ and 
\begin{equ}[e:wcbreakdown]
\E[F^2] = \sum_{n=0}^\infty n! \| f_n\|_{\fock_n}^2\, .
\end{equ}
Since the Hilbert space on which $\mu$ is defined is $\dot H^1(T)$ (and not $L^2(T)$ as in~\cite{Cannizzaro.etal2D2021}),
the isomorphism $I$ must be handled with care and the results in~\cite[Ch. 1.1.2]{nualartMalliavinCalculusRelated2006} 
applied accordingly. In particular, in the present context~\cite[Proposition 1.1.3]{nualartMalliavinCalculusRelated2006} 
translates as follows. Let $f\in\fock_n$ and $g\in \fock_m$, then 
\begin{equ}[e:ProdI]
I_n(f) I_m(g)=\sum_{p=0}^{n\wedge m} p! {n \choose p}{m \choose p}I_{m+n-2p}(f\otimes_{p}g)
\end{equ}
where 
\begin{equ}[e:contraction]
f\otimes_{p}g(x_{1:m+n-2p})\eqdef\int_{T^p} \qvar{\nabla_{y_{1:p}} f(x_{1:n-p},y_{1:p}), \nabla_{y_{1:p}} g(x_{n-p+1:m+n-2p},y_{1:p})} \dd y_{1:p}
\end{equ}
and $\qvar{\cdot,\cdot}$ denotes the usual scalar product in $\R^p$, the gradient $\nabla_{y_{1:p}}$ 
is only applied to the variables 
$y_{1:p}$ and, as in~\eqref{e:Hsnorm}, $x_{1:n}= (x_1,\dots,x_n)$. 
\medskip

%\begin{remark}\label{rem:WIintegral}
%The isomorphism $I_n$ is the analog of the $n$-th iterated Wiener-Ito integral 
%of~\cite[Ch. 1.1.2]{nualartMalliavinCalculusRelated2006} but with respect to isonormal 
%Gaussian process $\mu$. To see more explicitly the relation between the two, 
%note that $\mu$ can be represented as 
%\begin{equ}[e:imsw]
%\mu\eqdef (-\Delta)^{1/2}\eta
%\end{equ}
%for $\eta$ a space white noise on $T$, i.e. a Gaussian process whose covariance 
%is as in~\eqref{e:mu} but with $\dot H^1(T)$ replaced by $\dot H^0(T)=L^2(T)$. 
%Then, for any $f\in \dot H^1(T^n)$, we have 
%\begin{equ}[e:WIintegral]
%I_n(f)=W_n((-\DDelta_n)^{1/2}f)\eqdef\int_{T^n} (-\DDelta_n)^{1/2}f(y_1,\dots,y_n) W(\dd y_1,\dots,\dd y_n)%\eqdef \int \prod_{j=1}^n |k_i| \hat f(k_{1:n}) \hat W(\dd k_{1:n})
%\end{equ}
%\jacek{I like this notation a lot but I think its a bit unclear what $\cF((-\DDelta_i)^{1/2}f)(k_1,\dots, k_n) $ means for $i \neq n$, since $f$ is i guess its simply half-laplacians of $i$ number of terms?}
%where $(-\DDelta)^{1/2}f$ is defined via 
%\begin{equ}
%\cF((-\DDelta_n)^{1/2}f)(k_1,\dots, k_n)\eqdef \Big(\prod_{i=1}^n |k_i|\Big) \hat f(k_1,\dots,k_n)
%\end{equ}
%and $W$ is the $n$-dimensional Wiener measure (see~\cite[Proposition 1.1.4]{nualartMalliavinCalculusRelated2006}). 
%The identification~\eqref{e:WIintegral} is convenient in that, in what follows, we are able to directly 
%use the results of~\cite[Ch. 1]{nualartMalliavinCalculusRelated2006}. 
%\end{remark}

We say that $F: \cS'(T) \rightarrow \R$ is a \textit{cylinder function} if there exist $\phi_1, \dots, \phi_n \in \cS(T)$
and a smooth function $f: \R^n \rightarrow \R$ whose partial derivatives grow at most polynomially at infinity
such that $F[u] = f(u(\phi_1), \dots , u(\phi_n) )$. %We denote the set of cylinder functions by $\mathcal{C}$.
A random variable $F \in L^2 (\mu)$ is said to be smooth if it is a cylinder function on $\cS'(T)$ 
endowed with the measure $\mu$, i.e. there exist $\phi_1, \dots, \phi_n \in H$
and $f : \R^n \rightarrow \R$ as above such that $F= \f{\mu}$.
The \textit{Malliavin derivative} of a smooth random variable $F= \f{\mu}$ is the $H$-valued random variable 
given by
\begin{equ}[e:MalliavinD]
	DF \eqdef \sum_{i=1}^n \partial_i \f{\mu} \phi_i\, ,% \qquad\text{for all $x\in T$.}
\end{equ}
and we will denote by $D_xF$ the evaluation of $DF$ at $x$ and by $D_k F$ its Fourier transform at $k$. 
%In particular for $h \in H$ we have \giuseppeText{this is not what we will use later on...}
%\begin{equ}[e:MalDerivh]
%	 D_h F\eqdef \qvar{DF, h}_H =  \sum_{i=1}^n \partial_i \f{\eta} \qvar{\phi_i,h}_H   \, .
%\end{equ}
A commonly used tool in Wiener space analysis is 
\textit{Gaussian integration by parts}~\cite[Lemma 1.2.2]{nualartMalliavinCalculusRelated2006} 
which states that for any two smooth random variables $F,G \in L^2 (\mu)$ we have
\begin{equ}[e:ibp]
	\Ex [G \langle D F, h\rangle_{H} ] = \Ex [ -F\langle D G, h\rangle_{H} + FG \mu (h) ]\, .
\end{equ}
When on the torus, we will mostly work with the Fourier transform of $\mu$ 
which is a family of complex valued Gaussian random variables. 
Even though, strictly speaking, the results above do not cover this case, 
in \cite[Section 2]{Cannizzaro.etal2D2021} 
it was shown that one can naturally extend $\dot{H}^0 (\T^2, \R) = L^2 (\T^2,\R)$ to $L^2(\T^2, \C)$. 
Such extension can also be performed in the present context, 
and we are therefore allowed to exploit~\eqref{e:ibp} also in case of complex-valued $h$. 

\section{Invariant measures of the regularised equation}\label{s:invmeas}

The goal of this section is to construct a stationary solution to the regularised critical Navier-Stokes equation on $\R^2$.  
We will first consider the analogous equation on the torus of fixed size, where invariance is easier to obtain.
Subsequently, via a compactness argument, we will scale the size of the torus to infinity and 
characterise the limit of the corresponding solutions via a martingale problem.

 \subsection{The regularised Vorticity equation on $\T_M^2$}\label{s:SNSperiodic}

For $\theta\in(0,1]$, we consider the periodic version on $\T_M^2$ of~\eqref{e:vsnsN} given by
\begin{equ}[e:vsnsNT]
	\partial_t \tw =  -\tfrac12(-\Delta)^{\theta} \tw - \l \cN^{N,M}[\tw] + (-\Delta)^{\frac{\theta+1}{2}} \xi^M\,, \quad \tw(0,\cdot) = \omega_0^M \,,
\end{equ}
where $\omega_0^M$ is the initial condition, $\xi^M$ is a space-time white noise on $\R\times\T^2_M$ 
and $\cN^{N,M}$ is the non-linearity defined in \eqref{e:SNSnonlin}. 
% and $\mu^M$ is a Gaussian spatial noise on $\T_M^2$ which satisfies
%%\begin{equ}
%%    \Ex[\swn^M (\phi) \swn^M (\psi) ]= \qvar{\phi, \psi}_{L^2(\T_M^2)}  \, ,
%%\end{equ}
%%in particular
%\begin{equ}
%	\label{muM}
%	\Ex[\mu^M (\phi) \mu^M (\psi) ] = \qvar{\phi, \psi }_{\dot{H}^1(\T_M^2)}\, .
%\end{equ}
In Fourier variables,~\eqref{e:vsnsNT} becomes
\begin{equs}
	\dd\,\ftw_k = -\tfrac12 |k|^{2\theta} \ftw_k - \l\cN^{N,M}_k [\tw] + |k|^{\theta+1} \dd B_k(t)\,, \qquad k\in\Z^2_M\,
\end{equs}
where the complex-valued Brownian motions $B_k$ are defined via $B_k(t)\eqdef\int_0^t\hat\xi^M_k(\dd s)$, 
$\hat\xi^M_k$ being the $k$-th Fourier mode of $\xi^M$, and the Fourier transform of 
the non-linearity $\cN^{N,M}$ takes the form%\footnote{see Appendix~\ref{app:Derivation} for a justification of the equality below.}
\begin{equ}[e:nonlintorus]
	\cN^{N,M}_k [\tw] = \frac{1}{M^2}\sum_{\ell + m = k} \mathcal{K}_{\ell,m}^{N} \tw_{\ell} \tw_{m} \, ,
\end{equ}
for
\begin{equ}[e:nonlintorusCoeff]
	\quad \K{\ell,m} \eqdef \frac{1}{2\pi} \molf_{\ell,m}  \frac{(\ell^\perp\cdot (\ell+m))(m\cdot (\ell+m))}{|\ell|^2|m|^2} \,,\qquad\text{with}\quad \molf_{\ell,m}\eqdef \molf_\ell\molf_m\molf_{\ell+m}
\end{equ}
and the variables $\ell$ and $m$ appearing in the previous equations range over $\Z^2_M$. 
%while in~\eqref{e:nonlintorusCoeff}, to shorten the notation, we set $\molf_{\ell,m}\eqdef \molf(\ell)\molf(m)\molf(\ell+m)$.

As a first step in our analysis, we determine  basic properties of the
solution of~\eqref{e:vsnsNT}.

\begin{proposition}\label{p:ExistenceMarkov}
Let $M , N\in\N$ and $\theta\in(0,1]$. Then, for every deterministic initial condition 
$\tw_0\in \dot{H}^{-2}(\T^2_M)$,~\eqref{e:vsnsNT} has a unique
strong solution $\tw\in C(\R_+,\dot{H}^{-2}(\T^2_M))$. Further, $\tw$ is a strong Markov process.
\end{proposition}
\begin{proof}
The regularisation of the non-linearity is chosen in such a way that the first $N$ Fourier modes 
of $\omega^{N,M}$ are decoupled from $\{\omega_k^{N,M}\}_{|k|\geq N}$. Now, the latter 
is an Ornstein–Uhlenbeck process which is well-known to belong to $ C(\R_+,\dot{H}^{-2}(\T^2_M))$. 
The former instead solves a non-linear SPDE that preserves the $\dot{H}^{-1}$ norm 
as shown in Lemma \ref{l:H-1invariance} below. 
The conclusion can therefore be reached arguing 
as in~\cite[Section 7]{gubinelliRegularizationNoiseStochastic2013} 
(see also~\cite[Proposition 3.4]{Cannizzaro.etal2D2021}). 
\end{proof}

\begin{lemma}\label{l:H-1invariance}
Let $T=\T^2_M$ or $\R^2$.  Then for any distribution $\mu \in\CS'(T)$ such that 
$\grad \cdot (K\ast (\mu \ast \mol)) = 0$ we have
	\begin{equ}
		\langle \nnlin[\mu],\mu\rangle_{\dot{H}^{-1}(T)}=0 \, .
	\end{equ}
\end{lemma}
\begin{proof}
Let $\psi^{N} = K \ast (\mu \ast \mol)$ so that $\grad \cdot \psi^{N} =0$. 
Since $N$ is fixed throughout the proof, we will omit the superscript of $\psi$. 
Notice first that 
\begin{equs}
\qvar{\nnlin[\mu],\mu}_{\dot H^{-1} (T)} &= \qvar{\grad^{\perp} \cdot (\psi \cdot \grad \psi), \mu}_{\dot H^{-1} (T)} = \qvar{\grad^{\perp} \cdot (\psi \cdot \grad \psi), \grad^{\perp} \cdot \psi }_{\dot H^{-1} (T)} \\
		&= \qvar{\psi \cdot \grad \psi, (-\Delta) \cdot \psi }_{\dot H^{-1} (T)} = \qvar{\psi \cdot \grad\psi, \psi }_{L^2(T)}\,.
\end{equs}
The result now follows since 
the first term in the last scalar product at the right hand side is nothing 
but the Navier-Stokes non-linearity (see~\eqref{e:sns})) 
for which the equality is well-known. (Alternatively, one can perform a simple integration by parts 
and exploit the divergence free assumption $\grad \cdot \psi = 0$.) 
\end{proof}

Even though the generator $\tgen$ of the Markov process $\tw$, is a complicated operator,
its action on cylinder functions $F$ can be easily obtained by applying It\^o's formula and singling out the
drift term. By doing so, we deduce that for any such $F$, $\tgen F$ can be written as $\tgen F= \CL_\theta^M F+ \tgena F$,
where $\CL_0^M $ and $\tgena$ are given by
\begin{equs}[e:SNSgent]
	\CL_\theta^M F(\w) &\eqdef \half \sum_{i=1}^n \w (-(-\Delta)^{\theta} \phi_i )  \partial_i f  +  \half\sum_{i,j=1}^n \qvar{ \phi_i, \phi_j}_{\dot H^{\theta+1}(\T^2_M)}\, \partial^2_{i,j} f,\\
	\tgena F(\w) &\eqdef - \l \sum_{i=1}^n  \cN^{N,M} [\w](\phi_i)\, \partial_i f,
\end{equs}
where we abbreviated $\partial_i f=\partial_i \f{\w}$. We are now ready to prove the following proposition.

\begin{proposition}
	\label{p:SNStorinvar}
Let $\mu^M$ be the Gaussian spatial noise on $\T_M^2$ with covariance given by~\eqref{e:mu}. 
%\begin{equ}[muM]
%    \Ex[\mu^M (\phi) \mu^M (\psi) ]= \qvar{\phi, \psi}_{\dot H^1(\T_M^2)}\,,\qquad \text{for all $\phi, \psi\in H^1(\T_M^2)$.}
%\end{equ}
Then, for every $\theta\in(0,1]$, $\mu^M$ is an invariant measure of the solution $\tw$ of \eqref{e:vsnsNT}.
\end{proposition}
\begin{proof}
	The proof of this statement follows the steps of~\cite[Section 7]{gubinelliRegularizationNoiseStochastic2013} but we
	provide it here for completeness.
By Echeverr\`ia's criterion~ \cite{echeverriaCriterionInvariantMeasures1982}, it suffices to show that for any cylinder function
	$F=f(\mu^M(\phi_1),\dots,\mu^M(\phi_n))$ with $f$ at least twice continuous differentiable,
	we have $\E [ \tgen F (\mu^M) ] = 0$, where $\E$ is the expectation taken with respect to the law of $\mu^M$.
	Since, throughout the proof $M$ is fixed, we will omit it as a superscript to lighten the notation. 
	We will use the Fourier representation of the operators $\CL_\theta^M$ and $\tgena$, 
	which can be deduced by~\eqref{e:SNSgent} simply taking $F$ depending on (finitely many) Fourier modes of $\mu$ 
	and is 
	\begin{equs}
	\CL_\theta^M F(\mu) &= \frac{1}{2M^2} \sum_{k } |k|^{2\theta} \left( -\mu_{-k} D_k + |k|^2 D_{-k}D_k  \right) F(\mu)\,,\label{e:SNSgenst}\\
	\tgena F(\mu) &%= -\qvar{\l \nnlin[v], DF [v] } 
	= -\frac{\l}{M^4}\sum_{i,j }   \K{i,j} \mu_i \mu_j  D_{-i-j} F(\mu)\,.\label{e:SNSgenat}
\end{equs}
	Let  us first show that $\E [\CL_\theta^M F (\mu) ] = 0$.
	Let $k\in\Z^2_M$. Exploiting $|k|^2 e_k = (-\Delta)e_k$ and applying Gaussian integration by parts~\eqref{e:ibp}
	with $h = e_k$, $G = 1$ and $F = D_k F$, we obtain
	\begin{equs}
		\E  [|k|^2 D_{-k} D_k F(\mu)] %&= \E ( \qvar{ D (D_k F[\mu^M]) , (-\Delta) e_k }_{L^2} ) 
		= \E [\qvar{D (D_k F(\mu)) , e_k }_{\dot{H}^1} ] = \E[\mu^M_{-k} D_k F (\mu)]
	\end{equs}
	which immediately implies $\E [\CL_\theta F(\mu) ] =0$. We now turn to $\E [\tgena F (\mu) ]$.
	Let $i,\,j\in\Z^2_M$ such that $i+j\neq 0$. We apply once more Gaussian integration by parts,
	this time choosing $G= \mu_i \mu_j$ and $h = e_{i+j}$, so that we have
	\begin{equs}
		\E [\mu_i \mu_j D_{-i-j} F (\mu) ] &= -\frac{1}{|i+j|^2}\E [\mu_i \mu_j \qvar{ DF(\mu), e_{i+j} }_{\dot{H}^1(\T^2_M)}  ]\\
		&= -\frac{1}{|i+j|^2} \E[-F(\mu) \qvar{D(\mu_i \mu_j) ,e_{i+j}}_{\dot{H}^1(\T^2_M)} + \mu_i \mu_j \mu_{-i-j} F(\mu)]\\
		&= \E[-F(\mu) D_{-i-j}(\mu_i \mu_j) ] - \E\left[\left(\frac{1}{|i+j|^2}\mu_i \mu_j \mu_{-i-j}\right)F(\mu)\right]
	\end{equs}
	Now, $D_{-i-j}(\mu_i \mu_j)\neq 0$ if and only if either $i$ or $j$ are $0$ in which case
	$\K{i,j}$ in~\eqref{e:nonlintorusCoeff} is $0$.
	Hence, the first summand above does not contribute to $\E [\tgena F (\mu) ]$ and we obtain
	\begin{equs}
		\E [\tgena F (\mu) ] &= \E\left[\left(-\frac{\l}{M^4}\sum_{i,j}\frac{\K{i,j}}{|i+j|^2}\mu_i \mu_j \mu_{-i-j}\right)F(\mu)\right] \\
		&=\E\left[\qvar{\cN^{N,M}[\mu],\mu}_{\dot{H}^{-1}(\T^2_M)}F(\mu)\right]=0
	\end{equs}
	where the last equality follows by Lemma~\ref{l:H-1invariance} so that the proof is concluded.
\end{proof}

From now on, we will only work with the stationary solution of~\eqref{e:vsnsNT}, i.e. 
the initial condition will always be taken to be 
\begin{equ}[e:IC]
\omega_0^{N,M}\eqdef\mu^M
\end{equ}
where $\mu^M$ is as in Proposition~\ref{p:SNStorinvar}. 

In the following statements, we aim at obtaining estimates on the solution $\tw$ to~\eqref{e:vsnsNT}
which are uniform in both $N$ and $M$. A crucial tool is the so-called It\^o's trick, 
first introduced in~\cite{gubinelliRegularizationNoiseStochastic2013}. 
To the reader's convenience, we now recall its statement,
adapted to the present context.

\begin{lemma}[It\^o-Trick]
	\label{l:itotrick}
	Let $\theta\in(0,1]$. Let $\cL^{N,M}$ be the generator of the Markov process $\tw$, solution to~\eqref{e:vsnsN} started from
	the invariant measure $\mu^M$ in~\eqref{e:IC}, and $\cL_\theta^M$ and $\tgena$
	be defined according to~\eqref{e:SNSgent}. Let $T>0$ and $F$ a cylinder function on $\CS'(\T^2_M)$.
	Then, for every $p\geq 2$, there exists a constant $C>0$ depending only on $p$ such that
	\begin{equs}[e:itotrick]
		\mathbf{E} \left[ \sup_{t \leq T} \Big| \int_0^t \CL_\theta^M F(\tw_s) \dd s \Big|^p \right]^{1/p} \leq C  T\phalf \E \left[ \mathcal{E}(F) \right]^{1/2}\, ,
	\end{equs}
	where the energy $\mathcal{E} (F)$ is given by
	\begin{equ}[e:SNSTenergy]
		\mathcal{E}^{M}(F)(\mu^M) \eqdef \frac{1}{M^2} \sum_{k \in \Z_M^2 } |k|^{2+2\theta}  |D_k F(\mu^M ) |^2=\int_{\T^2_M} |(-\Delta_x)^{\frac{1+\theta}{2}} D_x F(\mu^M)|^2\dd x\,,
	\end{equ}
	the laplacian above clearly acting on the $x$ variable. % (recall that $D F\in \dot H^1(\T^2_M)$ by~\eqref{e:MalliavinD}). 
	Here and throughout, $\mathbf{E}$ denotes the expectation with respect to the law of the process
	$\{\tw_t\}_{t\in[0,T]}$, while $\E$ that with respect to the invariant measure $\mu^M$.
\end{lemma}
\begin{proof}
	The proof of the It\^o's trick can be found in several works,
	e.g.~\cite{gubinelliRegularizationNoiseStochastic2013,gubinelliEnergySolutionsKPZ2018}.
	With respect to those references, we further exploited Gaussian hypercontractivity~\cite[Theorem 1.4.1]{nualartMalliavinCalculusRelated2006}
	to replace the $p/2$ moment at the right hand side of~\eqref{e:itotrick} 
	with the square-root of the expectation of the energy.
\end{proof}

The It\^o's trick allows us to upper-bound moments of the integral in time of certain functionals of $\tw$
in terms of the first moment of the energy $\CE$ with respect to the law of $\tw$ at fixed time. Such a law
is explicit and Gaussian making the bound particularly useful. % in the present setting.
In the following proposition, we determine suitable estimates on the non-linearity.

\begin{proposition}\label{p:EnergyEstimates}
	Let $\theta\in(0,1]$, $T>0$ be fixed and $p\geq 2$. For $M,\,N\in\N$, let $\cN^{N,M}$ be defined according to~\eqref{e:nonlintorus} and
	$\l$ be as in~\eqref{e:CouplingConstant}. Then, there exists a constant $C=C(p)>0$, independent of $M,\,N\in\N$
	such that for all $\phi\in\CS(\T_M^2)$ and all $t\in[0,T]$, we have
	\begin{equs}
		\mathbf{E} \left[ \sup_{s \leq t} \Big| \int_0^s \tw_r(-(-\Delta)^\theta\phi) \dd r \Big|^p \right]^{1/p} &\leq C  t\phalf \|\phi\|_{\dot{H}^{1+\theta}(\T_M^2)}\, ,\label{e:Delta}\\
		\mathbf{E} \left[ \sup_{s \leq t} \Big|\l \int_0^s \cN^{N,M}[\tw_r](\phi)] \dd r \Big|^p \right]^{1/p} &\leq C N^{\theta-1} (t\vee t\phalf) \|\phi\|_{\dot{H}^2(\T_M^2)}  \, .\label{e:nonlinbound}
	\end{equs}
\end{proposition}

The proof of the previous proposition (and in particular of~\eqref{e:nonlinbound}) is based on the following lemma. 

\begin{lemma}\label{l:CaosNnlin}
For $M,\,N\in\N$, $\phi\in\CS(\T^2_M)$, let
$\cN^{N,M}[\mu^M](\phi)$ be the smooth  random variable defined according to~\eqref{e:nonlintorus}, 
with $\mu^M$ replacing $\tw$.
Then, $\cN^{N,M}_k[\mu^M](\phi)$ belongs to the second homogeneous Wiener chaos $\wc_2$. 
%and is such that 
%\begin{equ}[e:SNSTwiener]
%\cN^{N,M} [\mu^M] (\phi)=W_2\Big(\frac{1}{M^2}\sum_{i,j \in \Z_M^2} \K{i,j} \phi_{i+j} \Big)
%\end{equ}
%where $:\eta_i \eta_j:$ is the Wick product of $\eta_i$ and $\eta_j$ and $\K{i,j}$ is as in~\eqref{e:nonlintorusCoeff}.
Further, for all $\theta\in(0,1]$ the Poisson equation
\begin{equ}[e:Poisson]
	(1 -\CL_\theta^M ) H^{N,M}[\mu^M] (\phi)= \l \cN^{N,M} [\mu^M] (\phi)
\end{equ}
has a unique solution whose energy satisfies
\begin{equ}[e:EnergyPoisson]
	\E[\mathcal{E}^{N,M} (H^{N,M}[\mu^M](\phi))] =\frac{4\l^2}{M^4}\sum_{\ell,m}|\ell|^{2+2\theta}|m|^2\frac{ (\K{\ell,m})^2}{( 1+ \frac12 (|\ell|^{2\theta}+|m|^{2\theta}))^2} |\phi_{-\ell-m}|^2\,.
\end{equ}
\end{lemma}
\begin{proof}
Note that, by~\eqref{e:SNSnonlin} the non-linearity $\cN^{N,M}$ tested against $\phi$ 
can be written as 
\begin{equ}[e:NonlinNew]
\cN^{N,M}[\mu^M](\phi)=- \langle \mu^M(K\ast\mol_{\cdot}) \,\mu^M(\mol_{\cdot})\,,\, \nabla \phi\ast\mol_{\cdot} \rangle\,,
\end{equ}
the scalar product at the right hand side being the usual $L^2$ pairing.
Now, thanks to our choice of the mollifier $\rho$ in~\eqref{ass:rho}, and in particular the fact that 
its Fourier transform is $0$ in a neighbourhood of the origin, both $K\ast\mol_{\cdot}$ and $\mol_{\cdot}$ 
live in $\cS(\T^2_M)$ so that the expectation of the right hand side of~\eqref{e:NonlinNew} 
is finite by~\eqref{e:mu}. Hence, further using translation invariance, we have 
\begin{equ}
\E[\cN^{N,M}[\mu^M](\phi)]=\langle \E[\mu^M(K\ast\mol) \,\mu^M(\mol)], \nabla \phi\ast\mol \rangle\rangle=\langle K\ast\mol, \mol\rangle_{\dot H^1(\T^2_M)}\langle1, \nabla \phi\ast\mol \rangle\,,
\end{equ} 
which is zero since, by integration by parts, $\langle1, \nabla \phi\ast\mol \rangle=0$.  
Now, $\cN^{N,M} [\mu](\phi)$ is quadratic in $\mu$ and its component in the $0$-th chaos is $0$, hence $\nnlin[\mu^M](\phi)\in \wc_2$ and $\cN^{N,M}[\mu^M](\phi)=I_2(\fn^{N,M}_{\phi})$, 
for $\fn^{N,M}_{\phi}$ such that 
\begin{equ}[e:KernelNnlin]
\hat \fn^{N,M}_{\phi}(\ell, m)=\mathcal{K}_{\ell,m}^{N} \phi_{-\ell-m}\,.
\end{equ}
Let $\fh^{N,M}_{\phi}\in \fock_2$ and 
$H^{N,M}[\mu^M](\phi)=I_2(\fh^{N,M}_{\phi})$. 
Then, \cite[Lemma 2.3]{Gubinelli.TurraHyperviscous2020}  implies that 
\begin{equ}
(1-\CL_\theta^M) H^{N,M}[\mu^M](\phi) = (1-\CL_\theta^M) I_2(\fh^{N,M}_{\phi}) = I_2 \left((1 -\tfrac12(-\Delta)^\theta)\fh^{N,M}_{\phi}\right)\,.
\end{equ}
Equating the right hand side above and $\l I_2(\fn^{N,M}_{\phi})$, 
we immediately deduce that~\eqref{e:Poisson} has a unique solution 
which must necessarily satisfy
\begin{equ}[e:PoissonSol]
\hat \fh^{N,M}_{\phi}(\ell, m)=\l\frac{\mathcal{K}_{\ell,m}^{N}}{1 +\frac12(|\ell|^{2\theta}+|m|^{2\theta})}\phi_{-\ell-m}\,,\qquad \text{for all $\ell, m\in \Z^2_M$.}
\end{equ}
In order to compute the energy of $H^{N,M}[\mu^M](\phi)$, notice that by~\cite[Proposition 1.2.7]{nualartMalliavinCalculusRelated2006},
\begin{equ}
D_x H^{N,M}[\mu^M](\phi)= D_x I_2(\fh^{N,M}_{\phi})=2 I_1(\fh^{N,M}_{\phi}(x,\cdot))
\end{equ}
which implies, by linearity of $I_1$, 
\begin{equ}
\mathcal{E}^{N,M} (H^{N,M}[\mu^M](\phi))=4\int_{\T^2_M} \left| I_1 \left((-\Delta_x)^{\frac{1+\theta}{2}} \fh^{N,M}_{\phi}(x,\cdot)\right)\right|^2\dd x\,.
\end{equ}
Consequently, since $I_1$ is an isometry from $\wc_1$ and $\fock_1=\dot H^1(\T^2_M)$, we get 
\begin{equ}
\E[\mathcal{E}^{N,M} (H^{N,M}[\mu^M](\phi))]=4\int_{\T^2_M}\|(-\Delta_x)^{\frac{1+\theta}{2}} \fh^{N,M}_{\phi}(x,\cdot)\|^2_{\dot H^1(\T^2_M)}\dd x
\end{equ}
from which~\eqref{e:EnergyPoisson} simply follows by Plancherel's identity and~\eqref{e:PoissonSol}
\end{proof}

\begin{proof}[of Proposition~\ref{p:EnergyEstimates}]
For both~\eqref{e:Delta} and~\eqref{e:nonlinbound}, we will exploit the It\^o's trick Lemma~\ref{l:itotrick}. 
Let us begin with the former. Set $K^{N,M} [\mu^M] (\phi) \eqdef \mu^M (\phi)$, and 
notice that by~\eqref{e:SNSgent}, it is immediate to verify that 
\begin{equs}
	\CL_\theta^M K^{N,M} [\mu^M] (\phi) =\mu^M \left(-\tfrac12(-\Delta)^{\theta} \phi\right)\,,\quad\text{and}\quad \mathcal{E}^{N,M} (K^{N,M}[\mu^M](\phi) )=\| \phi \|_{\dot{H}^{1+\theta}(\T_M^2)}^2\,.
\end{equs}
Hence, the left hand side of~\eqref{e:Delta} equals
\begin{equs}[e:uTbound]
2\Ex \Big[ \sup_{s \leq t}\Big| \int_0^s \CL_\theta^M K^M [\omega_r^{N,M}] (\phi) \dd r \Big|^p \Big]^{1/p}\lesssim t\phalf  \| \phi \|_{\dot{H}^{1+\theta}(\T_M^2)}\, ,
\end{equs}
where in the last passage we applied~\eqref{e:itotrick}. 

We now turn to~\eqref{e:nonlinbound} for which we proceed similarly 
to~\cite[Proposition 3.15]{gubinelliEnergySolutionsKPZ2018}. 
Let $H^{N,M}$ be the solution to \eqref{e:Poisson} determined in Lemma \ref{l:CaosNnlin}. 
Then 
\begin{equs}[e:TwoSummands]
&\Ex \Big[ \sup_{s \leq t}\Big|  \int_0^s  \l\nnlin [\tw_r] (\phi) \dd r \Big|^p \Big]^{\frac1p}= \Ex \Big[ \sup_{s \leq t}\Big| \int_0^s  ( 1-\CL_\theta^M) H^{N,M} [\tw_r] (\phi) \dd r \Big|^p \Big]^{\frac1p}  \\
&\leq  \Ex \Big[ \sup_{s \leq t}\Big| \int_0^s  H^{N,M} [\tw_r] (\phi) \dd r \Big|^p \Big]^{\frac1p}+\Ex \Big[ \sup_{s \leq t}\Big| \int_0^s  \CL_\theta^M H^{N,M} [\tw_r] (\phi) \dd r \Big|^p \Big]^{\frac1p} \, .
\end{equs}
We will separately estimate the two summands above. 
For the second, we apply once more~\eqref{e:itotrick}, which, together with~\eqref{e:EnergyPoisson}, gives 
\begin{equs}
\Ex \Big[ \sup_{s \leq t}&\Big| \int_0^s  \CL_\theta^M H^{N,M} [\tw_r] (\phi) \dd r \Big|^p \Big]^{\frac2p}\lesssim t\frac{\l^2}{M^4}\sum_{\ell,m}|\ell|^{2+2\theta}|m|^2\frac{ (\K{\ell,m})^2 |\phi_{-\ell-m}|^2}{( 1+ \frac12 (|\ell|^{2\theta}+|m|^{2\theta}))^2} \\
&\lesssim t \frac{1}{M^2}\sum_{k}|k|^4|\phi_k|^2\frac{\l^2}{M^2}\sum_{\ell+m=k}(\molf_{\ell,m})^2\frac{|\ell|^{2\theta}}{( 1+ \frac12 (|\ell|^{2\theta}+|m|^{2\theta}))^2}\\
&\lesssim t \frac{1}{M^2}\sum_{k}|k|^4|\phi_k|^2\frac{\l^2}{M^2}\sum_{\ell}(\molf_{\ell})^2\frac{1}{1+ \frac12 |\ell|^{2\theta}}\leq t\|\phi\|^2_{\dot H^2(\T^2_M)}\frac{\l^2}{M^2}\sum_{|\ell|\leq N}\frac{1}{1+ \frac12 |\ell|^{2\theta}}
\end{equs}
where we  bounded $|\mathcal{K}_{\ell,m}^{N}|\leq \molf_{\ell} |\ell+m|^2/(|\ell||m|)$ and 
applied a simple change of variables. Now, the remaining sum can be controlled via
\begin{equs}
\frac{\l^2}{M^2}\sum_{|\ell|\leq N}\frac{1}{1+ \frac12 |\ell|^{2\theta}}\lesssim \l^2 \int_{|x|\leq N}\frac{\dd x}{1+\frac12|x|^{2\theta}}\lesssim 
\begin{cases}
\l^2 \log N\lesssim 1\,, & \text{if $\theta=1$,}\\
\l^2N^{2-2\theta}\lesssim N^{2\theta-2} \,, & \text{if $\theta\in(0,1)$,}
\end{cases}
\end{equs}
the last inequality being a consequence of~\eqref{e:CouplingConstant}. 

Let us turn to the first summand in~\eqref{e:TwoSummands}. We have 
\begin{equs}[e:firstterm]
\Ex \Big[ \sup_{s \leq t}&\Big| \int_0^s  H^{N,M} [\tw_r] (\phi) \dd r \Big|^p \Big]^{\frac1p}\leq  \Ex \Big[ \Big(\int_0^t  |H^{N,M} [\tw_r] (\phi)| \dd r \Big)^p \Big]^{\frac1p}\\
&\leq t^{1-\frac1p}\Ex \Big[\int_0^t  |H^{N,M} [\tw_r] (\phi)|^p \dd r  \Big]^{1/p}=t \E[|H^{N,M} [\mu^M] (\phi)|^p]^{\frac1p}\\&\lesssim t \E[|H^{N,M} [\mu^M] (\phi)|^2]^{\frac12}=\sqrt{2} t \|\fh^{N,M}_{\phi}\|_{\fock_2}
\end{equs}
where, from the first to the second line we used Jensen's inequality, from the second to the third 
Gaussian hypercontractivity~\cite[Theorem 1.4.1]{nualartMalliavinCalculusRelated2006} and 
the last step is a consequence of~\eqref{e:wcbreakdown} and the fact that, as shown in the proof of 
Lemma~\ref{l:CaosNnlin}, $H^{N,M} [\mu^M] (\phi)=I_2(\fh^{N,M}_{\phi})$ 
for $\fh^{N,M}_{\phi}$ satisfying~\eqref{e:PoissonSol}. 
In turn, the norm of $\fh^{N,M}_{\phi}$ can be estimated via
\begin{equs}
\|&\fh^{N,M}_{\phi}\|_{\fock_2}^2=\frac{\l^2}{M^4}\sum_{\ell,m\in\Z^2_M} |\ell|^2|m|^2\frac{(\mathcal{K}_{\ell,m}^{N})^2}{(1 +\frac12(|\ell|^{2\theta}+|m|^{2\theta}))^2} |\phi_{-\ell-m}|^2\\
&\lesssim \frac{\l^2}{M^2}\sum_{k\in\Z^2_M} |k|^4 |\phi_k|^2 \frac{1}{M^2}\sum_{\ell+m=k}\frac{(\molf_\ell)^2}{(1 +\frac12(|\ell|^{2\theta}+|m|^{2\theta}))^2}\\
&\lesssim \frac{\l^2}{M^2}\sum_{k\in\Z^2_M} |k|^4 |\phi_k|^2 \int_{|x|\leq N}\frac{\dd x}{(1+|x|^{2\theta})^2}\lesssim\|\phi\|_{\dot H^2(\T^2_M)}^2\times
\begin{cases}
\l^2\,,&\text{if $\theta>\tfrac12$,}\\
\l^2\log N \,,&\text{if $\theta=\tfrac12$,}\\
\l^2 N^{2-4\theta} \,,&\text{if $\theta<\tfrac12$,}\\
\end{cases}
\end{equs}
and, for any value of $\theta\in(0,1]$ the right hand side is bounded above by $N^{2\theta-2}\|\phi\|_{\dot H^2(\T^2_M)}^2$.  
\end{proof}

\subsection{The regularised fractional Vorticity equation on $\R^2$}

In this section, we study the regularised fractional vorticity equation~\eqref{e:vsnsN} on the full space $\R^2$.
Our goal is to show, on the one hand that, for $N\in\N$ fixed, it admits a solution and on the other
that such a solution has an invariant measure $\mu$ satisfying~\eqref{e:mu} 
and therefore complete the proof of Theorem~\ref{thm:introSNSstation}. 
%Let us remark that, as noted in~\cite[Remark 3.1-(2)]{Funaki.QuastelKPZ2015} 
%in the context of the one-dimensional KPZ equation, 
%for the latter purpose Echeverria's criterion~\cite{echeverriaCriterionInvariantMeasures1982}
%is not directly applicable because of the infinite-dimensional setting we are working on. 
%Instead, we will follow instead a different strategy first appeared in~\cite{Funaki.QuastelKPZ2015}.
\medskip

Throughout this section, $N\in\N$ will be fixed.  For $T>0$ and $\theta\in(0,1]$, we say that
$\omega^N\in C([0,T],\CS'(\R^2))$ is a {\it weak solution}
of~\eqref{e:vsnsN} starting at $\omega_0\in\CS'(\R^2)$ if for all $\phi\in\CS(\R^2)$ %the following equality holds
\begin{equ}[e:SNSweak]
	\omega_t^N(\phi)-\omega_0(\phi)=\half \int_0^t \omega_s^N(-(-\Delta)^\theta\phi)\dd s +\l \int_0^t \nnlin [\wN_s] (\phi)\dd s -
	M_t(\phi)\,.
\end{equ}
where $\nnlin$ is defined according to~\eqref{e:SNSnonlin} and $M_{\cdot}(\cdot)$ is a continuous Gaussian process
whose covariance is given by
\begin{equ}[e:Cov]
	\Ex[M_t(\phi) M_s(\psi)]=(t\wedge s) \langle \phi,\psi\rangle_{\dot{H}^{1+\theta}(\R^2)}\,,\qquad \phi,\psi\in \dot{H}^{1+\theta}(\R^2)
\end{equ}
(so that, formally, ``$M_t(\phi)=\int_0^t \xi (\dd s, (-\Delta)^{\frac{1+\theta}{2}}\phi)$'' 
for a space-time white noise $\xi$ on $\R_+\times\R^2$).
Further, if $\omega_0$ is distributed according to $\mu$ in~\eqref{e:mu}, 
then we will say that the solution is {\it stationary}.

Let us introduce the operator $\ngen$ which is nothing but the $\R^2$ counterpart of $\tgen$ in~\eqref{e:SNSgent} 
and formally represents the generator of~\eqref{e:vsnsN}. 
%and can be derived from~\eqref{e:SNSweak} by applying It\^o's formula to functions of the form $\f{\omega^N}$ 
%for $f$ smooth. 
Once again, it can be written as the sum of two operators, i.e. $\ngen=\cL_\theta+\ngena$, 
whose action of cylinder functions  $F(\omega)=\f{\omega}$  is given by 
\begin{equs}
	\cL_\theta F(\w) &\eqdef \half \sum_{i=1}^n \w (-(-\Delta)^{\theta} \phi_i )  \partial_i f  +  \half\sum_{i,j=1}^n \qvar{ \phi_i, \phi_j}_{\dot H^{1+\theta}(\R^2)}\, \partial^2_{i,j} f\,,\label{e:SNSgens}\\
	\ngena F(\w) &\eqdef - \l \sum_i^n \nnlin [\w](\phi_i)\,\partial_i f .\label{e:SNSgena}
\end{equs}
%
%\begin{equs}
%	\ngen F(\w) \eqdef&  \half \sum_{i,j} \partial_{ij} \f{\w }  \qvar{(-\Delta) \phi_i , (-\Delta) \phi_j} \\
%	&+\sum_i \partial_i \f{\w } \left[ \l \nnlin[\w ](\phi_i )  - \half \w ((-\Delta) \phi_i )  \right]\, .
%\end{equs}
Note that, thanks to the regularisation of the non-linearity (see Assumption~\ref{a:mol}), 
both $\cL_0 F[\w]$ and $\ngena F[\w]$ are well-defined 
for any cylinder function $F$. 

In the following definition, we present the martingale problem associated to $\ngen$.

\begin{definition}\label{d:martprob}
	Let $T>0$, $\Omega=C([0,T], \cS'(\R^2))$ and $\CG=\CB(C([0,T], \cS'(\R^2)))$ 
	the canonical Borel $\sigma$-algebra on it.
	Let $\theta\in(0,1]$, $N\in\N$ and $\mu$ be a random field on $\CS'(\R^2)$.
	We say that a probability measure $\mathbf{P}^N$ on $(\Omega,\CG)$ {\it solves the cylinder martingale problem for $\ngen$
			with initial distribution $\mu$}, if for all cylinder functions $F$
	the canonical process $\omega^N$ under $\mathbf{P}^N$ is such that
	\begin{equ}[e:Mart]
		\cM_t(F)\eqdef F(\omega_t^N) - F(\mu) -\int_0^t \ngen F(\omega^N_s)\dd s
		%\\&\Gamma(F)_t\eqdef(\cM_t(F))^2-t\,  \|\Psi\|_{L^2(\T^2)}\label{e:Mart2}
	\end{equ}
	is a continuous martingale.
	%with quadratic variation 
	%\begin{equ}[e:QV]
	%\langle \cM_\cdot(F)\rangle_t=\int_0^t  \mathcal{E} (F[\wN_s])\dd s\,,\quad\text{for}\quad \mathcal{E} (F[\wN]) \eqdef \int_{\R^2} \big|(-\Delta_x) D_x F[\wN_s]\big|^2 \dd x 
	%\end{equ}
	%where $D_x$ is the Malliavin derivative defined in~\eqref{e:MalliavinD}. 
\end{definition}

As a first result, we determine the connection between the martingale
problem in Definition~\ref{d:martprob} and weak solutions of~\eqref{e:vsnsN}.

\begin{proposition}\label{p:SNSstation}
	Let $\theta\in(0,1]$, $N\in\N$ and $\mu$ be a random field on $\CS'(\R^2)$.
	Then, $\mathbf{P}^N$ is a solution to the cylinder martingale problem for $\ngen$ with initial distribution $\mu$
	if and only if the canonical process $\omega^N$ under $\mathbf{P}^N$ is a weak solution of~\eqref{e:vsnsN}. 
\end{proposition}
\begin{proof}
%	Notice first that, arguing as in~\cite[Section 7]{Gubinelli.TurraHyperviscous2020}, $\ngen F$ is well-defined
%  for every cylinder function $F$. Indeed, we can define $\ngen F$ by its fourier representation which we as know $\ngen = \ngenap + \ngenam + \ngens$ and the corresponding reperentations for each term is given in Lemma \ref{l:fouriers} with integral replaced by sum. %\giuseppeText{Add a discussion on why this is the case.}
	Notice first that if $\w^N$ is a weak solution of~\eqref{e:vsnsN}, then for any cylinder function $F$, 
	the right hand side of~\eqref{e:Mart} is a martingale by It\^o's formula. Hence, the law of $\w^N$ solves the martingale 
	problem of Definition~\ref{d:martprob}. 
	In order to show that the converse also holds, we follow the strategy of~\cite[Lemma 2.7]{Funaki.QuastelKPZ2015}.
	Let $\mathbf{P}^N$ be a solution to the martingale problem and $\wN$ the canonical process with respect to $\mathbf{P}^N$.
	Let $\phi\in\cS(\R^2)$ and $F_\phi$ be the linear cylinder function defined as $F_\phi(\wN)\eqdef \wN(\phi)$.
	In view of~\eqref{e:Mart}, $\wN$ satisfies
	\begin{equs}[e:Fphi]
		\wN_t(\phi)-\mu(\phi)&=\int_0^t \ngen \w_s (\phi) \dd s +\cM_t(F_\phi)\\
		&=\half \int_0^t \omega_s^N(-(-\Delta)^\theta\phi)\dd s +\l \int_0^t \nnlin [\wN_s] (\phi)\dd s +\cM_t(F_\phi)
	\end{equs}
	the second step being a consequence of the definition of $\ngen$ in~\eqref{e:SNSgens} and~\eqref{e:SNSgena}, and
	where $\cM_t(F_\phi)$ is a continuous martingale. We are left to show that
	for all $\phi$, $\cM_t(F_\phi)$ is Gaussian
	and has covariance given as in~\eqref{e:Cov}.
	To do so, let $\phi,\,\psi\in\cS(\R^2)$,
	and consider the quadratic cylinder function $F_{\phi,\psi}(\wN)\eqdef \wN(\phi)\,\wN(\psi)$.
	Exploiting~\eqref{e:Mart} once more, we see that
	\begin{equ}[e:Fphipsi]
		\cM_t(F_{\phi,\psi})= \wN_t (\phi) \wN_t (\psi)  -\mu(\phi) \mu (\psi) - \int_0^t \ngen F_{\phi,\psi}(\wN_s) \dd s
	\end{equ}
	is a martingale. Let $b_s(\phi)\eqdef \ngen \wN_s (\phi)$ and 
	notice that~\eqref{e:SNSgens} and~\eqref{e:SNSgena} give
	\begin{equ}
		\ngen F_{\phi,\psi}(\wN_s) = \wN_s (\phi) b_s( \psi) +\wN_s (\psi) b_s( \phi) + \qvar{\phi ,\psi}_{\dot{H}^{1+\theta}(\R^2)}\,,
	\end{equ}
	which, once plugged into~\eqref{e:Fphipsi}, provides
	\begin{equs}[e:MartDecomp]
		\cM_t(F_\phi)\cM_t(F_\psi)&-t\qvar{\phi ,\psi}_{\dot{H}^{1+\theta}(\R^2)}\\
		=& \cM_t(F_{\phi,\psi}) - \int_0^t \Big(b_s(\psi) \delta_{s,t}\wN_\cdot (\phi) + b_s(\phi)  \delta_{s,t}\wN_\cdot (\psi)\Big)\dd s\\
		&\qquad - \mu(\phi) \cM_t (F_\phi) - \mu(\psi) \cM_t (F_\phi) + \int_0^t \int_0^t b_s(\phi) b_{\bar{s}} (\psi) \dd s \dd \bar{s}\\
		=& \cM_t(F_{\phi,\psi}) - \int_0^t \Big(b_s(\psi) \int_s^t\dd \cM_{\bar s}(\phi)+ b_s(\phi)  \int_s^t\dd \cM_{\bar s}(\phi)\Big)\dd s\\
		&\qquad - \mu(\phi) \cM_t (F_\phi) - \mu(\psi) \cM_t (F_\phi)\\
		=&\cM_t(F_{\phi,\psi}) - \int_0^t \Big(\int_{0}^{\bar s}b_s(\psi)\dd s\Big) \dd \cM_{\bar s}(\phi)+ \int_0^t \Big(\int_{0}^{\bar s}b_s(\phi)\dd s\Big) \dd \cM_{\bar s}(\psi)\\
		&\qquad - \mu(\phi) \cM_t (F_\phi) - \mu(\psi) \cM_t (F_\phi)
	\end{equs}
	where we introduced the notation $\delta_{s,t}f_\cdot\eqdef f(t)-f(s)$ and exploited~\eqref{e:Fphi} in the second equality.
	Now, all the terms at the right hand side are martingales so that, by definition, $t \qvar{\phi ,\psi}_{\dot{H}^{1+\theta}(\R^2)}$
	is the quadratic covariation of $\cM_t(F_\phi)$ and $\cM_t(F_\psi)$ and clearly~\eqref{e:Cov} holds. 
	For Gaussianity, taking $\psi=\phi$ in~\eqref{e:MartDecomp}, 
	we deduce that $\cM_t(F_\phi)$ is a continuous martingale with deterministic quadratic variation
  which, in view of~\cite[Theorem 7.1.1]{Ethier.KurtzMarkov2009}, implies that, for all $\phi$, $\cM_t(F_\phi)$ is
	Gaussian with independent increments so that the proof is concluded. 
\end{proof}

We now show that the martingale problem of Definition~\ref{d:martprob} starting from $\mu$ as in~\eqref{e:mu} 
admits a solution. Together with the previous result, this implies
the existence of a stationary weak solution to~\eqref{e:vsnsN} whose invariant measure is $\mu$
thus completing the proof of Theorem~\ref{thm:introSNSstation}.

\begin{theorem}\label{thm:InvR2}
	Let $N\in\N$ be fixed, $\theta\in(0,1]$ and $\mu$ the Gaussian process with covariance given by~\eqref{e:mu}.
	The cylinder martingale problem of Definition~\ref{d:martprob} for $\ngen$ with initial distribution $\mu$ 
	has a solution $\mathbf{P}^N$. Further, the canonical process $\w^N$ under $\mathbf{P}^N$ has invariant measure $\mu$.
\end{theorem}

The proof of the previous theorem exploits the Galerkin approximation $\w^{N,M}$ of~\eqref{e:vsns}
studied in the previous section. In the next lemma, we show that the sequence is tight in $M$ (for $N$ fixed).

\begin{lemma}\label{l:tightvsnsNT}
	Let $N\in\N$ be fixed, $\theta\in(0,1]$ and $T>0$. 
	With a slight abuse of notation, for all $M\in\N$, let $\omega^{N,M}$ denote the
	periodically extended version of the stationary solution to~\eqref{e:vsnsNT} on $\T^M$.
	Then, the sequence $\{\tw \}_{M \in \N} $ is tight in $C([0,T], \cS' (\R^2))$.
\end{lemma}
\begin{proof}
	Thanks to \cite{mitomaTightnessProbabilities1983},  it suffices to show that for all $\phi \in\cS ( \R^2)$, 
	the sequence $\{t\to\wN_t(\phi)\}_N$ is tight. To do so, we will exploit Kolmogorov's criterion, 
	for which we need to prove that there exist $\alpha>0$ and $p>1$ such that  
	for all $0\leq s<t\leq T$ we have
	\begin{equs}
		\label{e:SNStightM}
		\Ex \left[ | \tw_t(\phi) - \tw_s ( \phi) |^p  \right]^{1/p} \lesssim_\phi (t-s)^{\alpha} \, ,
	\end{equs}
	where the constant hidden into ``$\lesssim$'' depends on $\phi$. 
	Since $\tw$ is Markov and stationary, it is enough to show~\eqref{e:SNStightM} for $s=0$. 
	Notice first that, by construction,
	the time increment of $\tw$ satisfies
	\begin{equs}
		\tw_t&(\phi) - \mu^M ( \phi)\\
		&=\frac12\int_0^t \tw_s ( -(-\Delta)^\theta \phi) \dd s - \l \int_0^t\nnlin [\tw_s] (\phi) \dd s + \int_0^t \xi^M ( \dd s, c  \phi) 
	\end{equs}
	and we will separately focus on each of the terms at the right hand side.
	Gaussian hypercontractivity~\cite[Theorem 1.4.1]{nualartMalliavinCalculusRelated2006} and the definition of $\xi$
	imply that the last term can be bounded as
	\begin{equs}[e:SNSxibound]
		\Ex \Big[\Big|&\int_0^t \xi^M ( s, (-\Delta)^\frac{1+\theta}{2}   \phi) \dd s \Big|^p \Big]^{1/p} \lesssim  \Ex  \Big[\Big|\int_0^t \xi^M ( s, (-\Delta)^\frac{1+\theta}{2}  \phi)\dd s  \Big|^2 \Big]\phalf  \\
		&=  \Big(\int_0^t \qvar{(-\Delta)^\frac{1+\theta}{2} \phi , (-\Delta)^\frac{1+\theta}{2} \phi  }_{L^2 ( \T^2_M)}\dd s\Big)^{\frac12}  =  t^{\frac12}  \| \phi \|_{\dot{H}^{1+\theta} ( \T^2_M ) }\lesssim  t^{\frac12}  \| \phi \|_{\dot{H}^{1+\theta} ( \R^2 ) }\,,
	\end{equs}
	where in the last step, we simply used the fact that the $\dot{H}^{1+\theta} ( \T^2_M )$-norm 
	is simply a Riemann-sum approximation of the $\dot{H}^{1+\theta} ( \R^2 )$ norm. % and $\phi\in \dot{H}^2 ( \R^2 )$. 
	For the remaining two terms, we exploit Lemma \ref{p:EnergyEstimates} and the same argument as above. 
	Collecting what deduced so far, we see that~\eqref{e:SNStightM} holds for all $\phi$, 
	any $p\geq 2$ and $\alpha=1/2$, 
	so that, tightness of the sequence $\{\wN\}_N$ follows at once by Kolmogorov's criterion 
	and~\cite{mitomaTightnessProbabilities1983}. 
\end{proof}

We are now ready to complete the proof of Theorem~\ref{thm:InvR2}.

\begin{proof}[of Theorem~\ref{thm:InvR2}]
Let $\mathbf{P}^{N,M}$ denote the law of the periodically extended version of the 
stationary solution $\tw$ of~\eqref{e:vsnsNT} 
on $C([0,T],\cS'(\R^2))$. 
Since by Lemma~\ref{l:tightvsnsNT}, the sequence $\{\mathbf{P}^{N,M}\}_M$ is tight, we can extract, 
via Prokhorov's theorem, a weakly converging subsequence that, slightly abusing the notation, 
we will still denote by $\{\mathbf{P}^{N,M}\}_M$. Let $\mathbf{P}^N$ be its limit. 
Skorokhod's representation theorem ensures that we can realise the sequence on a proper 
probability space in such a way that 
$\{ \tw \}_M$ converges to $\w^N$, $\mathbf{P}^N$ almost surely in $C([0,T], \cS' (\R^2))$ as $M \rightarrow \infty$. 
We now want to show that $\mathbf{P}^N$ is a solution to the martingale problem for $\ngen$, which amounts to verify that
for any cylinder function $F$ the right hand side of~\eqref{e:Mart} is a continuous martingale. 

As a preliminary step, note that since $\tw\to\w^N$ almost surely in $C([0,T], \cS' (\R^2))$, then, for all $t$, 
$\tw_t\to\w^N_t$ almost surely in $\cS'(\R^2)$. By assumption, $\tw_t$ is distributed according to $\mu^M$ 
and $\mu^M$ converges to $\mu$. Hence $\w^N_t$ is distributed according to $\mu$. 
In other words, $\mu$ is an invariant measure for $\w^N$ and,
as $\mu$ is Gaussian, for any cylinder function $G$, $G(\w^N_t)$ has finite moments 
of all orders. 

Let $\phi_1,\dots,\phi_n\in\cS(\R^2)$ and $F(\omega)=\f{\omega}$ 
be a cylinder function on $\cS'(\R^2)$. By It\^o's formula, for all $t\in[0,T]$,
\begin{equ}
F(\tw_t)  - F(\mu^M)- \int_0^t \tgen F(\tw_s) \dd s\, ,
\end{equ}
is a square-integrable continuous martingale. Therefore, by standard martingale convergence arguments,
 the result follows once we show that 
\begin{equ}[e:Conv]
\Big(F(\tw_t)  - F(\mu^M)- \int_0^t \tgen F(\tw_s) \dd s\Big)- \Big(F(\w_t^{N})-F(\mu)  - \int_0^t \ngen F(\w_s^{N}) \dd s\Big)
\end{equ}
goes to $0$ in, say, mean square with respect to $\mathbf{P}^N$. 
We will first prove that~\eqref{e:Conv} converges to $0$ almost surely.
Since $\tw\to\w^N$ almost surely in $C([0,T], \cS' (\R^2))$, then almost surely for all $r\in[0,T]$ and $n\in\N$ both 
\begin{equs}[e:BasicConv]
\partial^{(n)}\f{\tw} &\to \partial^{(n)}\f{\w^N_r}\,,\\%\qquad\text{for any $n\in\N$, and }\\
\tw_r(-(-\Delta)^\theta\phi)&\to\w^N_r(-(-\Delta)^\theta\phi)%\label{e:BasicConv}
\end{equs} 
hold. Further, for every $i,j=1,\dots n$, 
$\qvar{\phi_i,\phi_j}_{\dot H^{1+\theta}(\T^2_M)}\to\qvar{\phi_i,\phi_j}_{\dot H^{1+\theta}(\R^2)}$ 
deterministically as the $\dot H^{1+\theta}(\T^2_M)$-norm is a Riemann-sum approximation 
of the $\dot H^1(\R^2)$-norm.
Hence, by the definitions of $\cL_0^M$ and $\cL_0$ in~\eqref{e:SNSgenst} and~\eqref{e:SNSgens} respectively, 
it follows that almost surely
\begin{equ}
F(\tw_r)\to F(\w^N_r),\quad r\in\{0,t\}\qquad\text{ and }\qquad\int_0^t\cL_0^MF(\tw_s)\dd s\to \int_0^t\cL_0 F(\w^N_s)\dd s\,.
\end{equ}
In light of~\eqref{e:BasicConv}, to show that the same convergence holds for the term containing
$\tgena F(\tw_r)$ and $\ngena  F(\w^N_r)$, 
it suffices to argue that almost surely, for all $i=1,\dots,n$ and $r\in[0,T]$, 
$\cN^{N,M}[\tw_r](\phi_i)\to \nnlin [\w^N_r](\phi_i)$. This in turn is a direct consequence of 
the representation~\eqref{e:NonlinNew} and the fact that the almost sure convergence of 
$\tw$ to $\w^N$ in $C([0,T], \cS' (\R^2))$ ensures that 
both $\tw(K\ast\mol_{\cdot})\to \wN(K\ast\mol_{\cdot})$ and $\tw(\mol_{\cdot})\to \wN(\mol_{\cdot})$.  
Indeed, our choice of the mollifier guarantees that Fourier transform of $\mol$ is supported away from the origin 
so that $K\ast\mol_{\cdot}\in \cS(\R^2)$. 

In conclusion,~\eqref{e:Conv} converges to $0$ almost surely. Moreover, 
each of its summands has finite moments of all orders as for all $r\in[0,T]$ the distribution of 
$\tw_r$ and $\w^N_r$ is Gaussian. 
Therefore, by the dominated convergence theorem,~\eqref{e:Conv} converges to $0$ 
in mean square and the proof is concluded. 
\end{proof}
%
%All the elements are now in place and we are ready to complete the proof of Theorem~\ref{thm:introSNSstation}. 
%
%\begin{proof}[of Theorem~\ref{thm:introSNSstation}]
%
%
%\end{proof}

\section{The Vorticity equation on the real plane}

Throughout this section, we will be working with a solution $\mathbf{P}^N$ of the martingale problem 
for $\ngen$ with initial distribution $\mu$, whose canonical process $\w^N$ is, by Proposition~\ref{p:SNSstation}, 
a stationary weak solution of the fractional regularised vorticity equation~\eqref{e:SNSweak} on $\R^2$. 

The goal is to control the behaviour of $\w^N$ in the limit $N\to\infty$. 
To do so, we first need to deepen our understanding
of the generator $\ngen$ and, in particular, determine how it acts on random variables in $L^2(\mu)$.

\subsection{The operator $\ngen$} \label{s:contgen}

This section is devoted to the study of the properties of the operator $\ngen$ on $L^2(\mu)$, 
($\mu$ being the Gaussian process with covariance~\eqref{e:mu}) which is given by the sum of 
$\cL_0$ and $\ngena$ defined in~\eqref{e:SNSgens} and~\eqref{e:SNSgena}, respectively. 
Recall that, as remarked in Section~\ref{sec:Mall}, there exists an isomorphism $I$ between $L^2(\mu)$ and 
the Fock space $\fock$. With a slight abuse of notation, from here on we will denote with the same symbol 
any operator $\CO$ acting on $L^2(\mu)$ and the corresponding operator acting instead on $\fock$, 
where by ``corresponding'' we mean any operator $\mathfrak O$ such that $\CO I(\phi)= I(\mathfrak O\phi)$ 
for all $\phi\in\fock$.

\begin{proposition}\label{p:SNSstraightoperators}
Let $\mu$ be the Gaussian process whose covariance function is given by~\eqref{e:mu}. 
Then, for any $\theta\in(0,1]$, the operator $\cL_\theta$ is symmetric on $L^2(\mu)$, and 
for each $n$, it maps $\wc_n$ to itself. Further, for any $f\in \fock_n$, 
$\cL_\theta f=-\tfrac12 (-\Delta)^\theta f$ so that the Fourier transform of the left hand side equals
\begin{equ}[e:gensf]
\cF(\cL_\theta f)(k_{1:n})=-\tfrac12|k_{1:n}|^{2\theta} \hat f(k_{1:n})\,,\qquad \text{for all $k_{1:n}\in (\R^2)^n$,}
\end{equ}
where $|k_{1:n}|^{2\theta}\eqdef |k_1|^{2\theta}+\dots+|k_n|^{2\theta}$. 
Instead, the operator $\ngena$ is anti-symmetric on $L^2(\mu)$ and it can be written as the sum 
of two operators $\ngenap$ and $\ngenam$, the first mapping $\wc_n$ to $\wc_{n+1}$ while the second 
$\wc_n$ to $\wc_{n-1}$. 
Moreover, the adjoint of $\ngenap$ is $-\ngenam$ and 
for any $f\in \fock_n$ 
the Fourier transform of their action on $f$ is given by 
\begin{align}
\F( \ngenap f ) (k_{1:n+1}) &= \l n \K{k_1,k_2} \hat{f} (k_1 + k_2, k_{3:n+1} )  \label{e:genapf} \\
\F( \ngenam f ) (k_{1:n-1} ) &= 2\l n (n-1)  \int_{\R^2} \molf_{p,k_1-p} \frac{(k_1^{\perp} \cdot p)(k_1\cdot (k_1-p))}{|k_1|^2} \hat f(p,k_1-p,k_{2:n-1})\dd p \label{e:genamf} 
\end{align}
where $\CK^N$ was defined in~\eqref{e:nonlintorusCoeff} and $k_{1:n+1}\in (\R^2)^{n+1}$. 
Strictly speaking the functions at the right hand side need to be symmetrised 
with respect to all permutations of their arguments.   
\end{proposition}
\begin{proof}
The properties of $\cL_\theta$, including~\eqref{e:gensf}, were shown in a number of references, 
see e.g.~\cite[Ch. 2.4]{GPlecnotes} or~\cite[Lemma 3]{Gubinelli.TurraHyperviscous2020}, and therefore we omit their proof. 
Concerning $\ngena$, let $F(\mu)=\f{\mu}$ be a generic cylinder function. By~\eqref{e:SNSgena}, we have 
\begin{equs}[e:A]
\ngena F(\mu)&=-\l \sum_i \nnlin[\mu](\phi_i)\partial_i f = -\l \nnlin [\mu] \Big( \sum_i \partial_i f \phi_i   \Big) \\
&=  -\l \nnlin [ \mu] (DF)=-\l\int_{\R^2} \nnlin [ \mu] (x) D_xF \dd x
\end{equs}
where we exploited the definition of the Malliavin derivative in~\eqref{e:MalliavinD}. 

Let us first show the decomposition in $\ngenap$ and $\ngenam$ in~\eqref{e:genapf} and~\eqref{e:genamf}, respectively. 
By polarisation it suffices to take $F(\mu)=I_n(f)$ 
for $f$ of the form $\otimes^{n}\phi$ and $\phi\in\dot H^1(\R^2)$. 
Note that the Malliavin derivative of $F$ satisfies
\begin{equ}
D_x F(\mu)=n I_{n-1}\big(\otimes^{n-1}\phi\big) \phi(x)
\end{equ}
(see e.g.~\cite[proof of Lemma 3.5]{Cannizzaro.etal2D2021}). Therefore, plugging the previous into~\eqref{e:A}, we get
\begin{equ}
\ngena F(\mu)=-\l\int_{\R^2} \nnlin [ \mu] (x) D_xF \dd x=-n \l \nnlin [ \mu] (\phi) I_{n-1}\big(\otimes^{n-1}\phi\big)\,.
\end{equ}
Arguing as in the proof of Lemma~\ref{l:CaosNnlin}, it is not hard to see that $\nnlin [ \mu] (\phi)\in \wc_2$ 
and $\nnlin [ \mu] (\phi)= I_2(\fn^{N}_{\phi})$, 
the Fourier transform of $\fn^{N}_{\phi}$ being given by the right hand side of~\eqref{e:KernelNnlin} 
(though for $\ell,m\in\R^2$). 
Therefore, 
\begin{equs}[e:AIexp]
\ngena F(\mu)= -n \l I_2(\fn^{N}_{\phi}) I_{n-1}\big(\otimes^{n-1}\phi\big)=&-n\l I_{n+1}(\fn^{N}_{\phi}\otimes_0 \otimes^{n-1}\phi)\\
& -2n(n-1)\l I_{n-1}(\fn^{N}_{\phi}\otimes_1 \otimes^{n-1}\phi) \\
&-n(n-1)(n-2)\l I_{n-3}(\fn^{N}_{\phi}\otimes_2 \otimes^{n-1}\phi)
\end{equs}
where the last equality is a consequence of~\eqref{e:ProdI}. 
It is not hard to see, by taking Fourier transforms and applying Plancherel's identity, that the 
first term indeed equals $\ngenap I_n(f)$, while the second $\ngenam I_n(f)$, so that in particular 
$\ngenap$ and $\ngenam$ map $\wc_n$ into $\wc_{n+1}$ and $\wc_{n-1}$ respectively. 
We claim that instead the last term vanishes. Indeed by~\eqref{e:contraction}, we have 
\begin{equ}
\fn^{N}_{\phi}\otimes_2 \otimes^{n-1}\phi(x_{1:n-3})=\prod_{i=1}^{n-3}\phi(x_i)\int_{(\R^2)^2}\qvar{\nabla \fn^{N}_{\phi}(x,y), \nabla \phi(x)\phi(y)}\dd x\dd y
\end{equ}
Applying Plancherel's identity and the definition of $\fn^{N}_{\phi}(x,y)$, we see that the integral above equals
\begin{equs}
&\int_{(\R^2)^2}|k_1|^2|k_2|^2\hat\fn^{N}_{\phi}(k_1,k_2)\phi_{k_1}\phi_{k_2}\dd k_1\dd k_2\\
&=\int_{(\R^2)^2}|k_1|^2|k_2|^2\mathcal{K}_{k_1,k_2}^{N} \phi_{-k_1-k_2} \phi_{k_1}\phi_{k_2}\dd k_1\dd k_2%=\qvar{\cN^N[(-\Delta\phi)], (-\Delta)\phi}_{\dot H^1(\R^2)} 
=\qvar{\cN^N[(-\Delta\phi)], (-\Delta)\phi}_{\dot H^{-1}(\R^2)} % we multiply and divide by |k_1+k_2|^2 so that each phi has a square of some k's and then we have 1\|k_1+k_2|^2 remaining which gives the H^{-1}
\end{equs}
%and the right hand side is equal to $0$ by Lemma~\ref{l:CaosNnlin}. 
and the right hand side is equal to $0$ by Lemma~\ref{l:H-1invariance}. 
% At last we need to prove that the adjoint of $\ngenap$ is $-\ngenam$. 
% This in turn is an immediate consequence of the anti-symmetry 
% of $\ngena$ and the fact that $\ngenap$ and $\ngenam$ map $\wc_n$ to $\wc_{n+1}$ and $\wc_{n-1}$ 
% respectively. 

We now show that $\ngenap$ is the adjoint of $-\ngenam$. 
For $F = \sum_n I_n(f_n)$ and $G = \sum_n I_n(g_n)$ we have 
\begin{equs}
  &\Ex \left[ \ngenap F G \right] = \sum_{n,m} \Ex \left[ I_{n+1} (\ngenap f_n) I_m(g_{m}) \right] = \sum_n (n+1)! \qvar{\ngenap f_n,g_{n+1}}_{\Gamma L^2_{n+1}}\\
  &\Ex \left[  F \ngenam G \right] = \sum_{n,m} \Ex \left[ I_{n} ( f_n) I_{m-1}(\ngenam g_{m}) \right] = \sum_n n! \qvar{ f_n,\ngenam g_{n+1}}_{\Gamma L^2_{n+1}}\,,
\end{equs}
which is a consequence of orthogonality of different Wiener-chaoses. Therefore, to prove that the two 
right hand sides above are indeed equal, it suffices to verify that 
\begin{equ}
(n+1)\qvar{\ngenap f_n,g_{n+1}}_{\Gamma L^2_{n+1}} = -\qvar{ f_n,\ngenam g_{n+1}}_{\Gamma L^2_{n}}\,.
\end{equ} 
By \eqref{e:genapf} (modulo permutations), the left hand side is given by
\begin{equs}[e:aplusasym]
	&4\pi \l n (n+1) \int \left(  \Pi_{i=1}^{n+1} |k_i|^2  \right) \K{k_1,k_2} \hat{f} (k_1 + k_2, k_{3:n+1} ) \hat{g}_{n+1}(k_{1:n+1}) \dd k_{1:n+1}\, .
\end{equs}
Then, by a simple change of variables the previous integral is
\begin{equs}
	&\int \left(\Pi_{i=1, i \neq 2}^{n+1} |k_i|^2 \right) |k_2^\prime - k_1|^2 \K{k_1,k_2^\prime - k_1} \hat{f} (k_2^\prime, k_{3:n+1} ) \hat{g}_{n+1}(k_1, k_2^\prime - k_1 , k_{3:n+1}) \dd k_1 \dd k_2^\prime \dd k_{3:n+1}\\
	&=\int \left(\Pi_{i=1}^{n} |k_i|^2 \right) \frac{|k_1 - p|^2 |p|^2}{|k_1|^2} \K{p,k_1 - p} \hat{f} (k_{1:n}) \hat{g}_{n+1}(p, k_1 - p , k_{2:n}) \dd p \dd k_{1:n}\\
	&=\frac{1}{2\pi} \int \left(\Pi_{i=1}^{n} |k_i|^2 \right)  \hat{f} (k_{1:n}) \int \molf_{p,k_1-p} \frac{(p^\perp \cdot k_1)(k_1 \cdot (k_1-p))}{|k_1|^2} \hat{g}_{n+1}(p, k_1 - p , k_{2:n}) \dd p \dd k_{1:n} \\
	&=-\frac{1}{2\pi} \int \left(\Pi_{i=1}^{n} |k_i|^2 \right)  \hat{f} (k_{1:n}) \int \molf_{p,k_1-p} \frac{(p \cdot k_1^\perp)(k_1 \cdot (k_1-p))}{|k_1|^2} \hat{g}_{n+1}(p, k_1 - p , k_{2:n}) \dd p \dd k_{1:n} \, ,
\end{equs}
from which the result follows. 
Further, as an immediate corollary of $(\ngenap)^*=-\ngenam$, we also deduce that $\ngena$ is antisymmetric 
so that the proof of the statement is completed. 
\end{proof}

\subsection{Tightness and upper bound}\label{s:TightUp}

Following techniques similar to those exploited in Section~\ref{s:SNSperiodic}, we establish tightness 
for the sequence $\{\w^N\}_N$ of solutions to the stationary regularised vorticity equation 
under assumption~\eqref{e:CouplingConstant}. 
For $\theta=1$, we also derive an order one upper bound on the integral in time of the non-linearity.

\begin{theorem}\label{t:tightvsnsN}
Let $\theta\in(0,1]$. 
For $N\in\N$, let $\w^N$ be a stationary solution to~\eqref{e:SNSweak} on $\R^2$ with coupling constant $\l$ 
chosen according to~\eqref{e:CouplingConstant}, started from the 
Gaussian process $\mu$ with covariance given by~\eqref{e:mu}. For $\phi\in\cS(\R^2)$ and $t\geq 0$,  
set 
\begin{equ}[e:IntNonlin]
\cB^N_t(\phi)\eqdef \l\int_0^t\nnlin_t[\w^N_s](\phi)\dd s\,.
\end{equ}
Then, for any $T>0$, the couple $(\w^N, \cB^N)$ is tight in the space $C([0,T],\cS'(\R^2))$. 
Moreover, for $\theta=1$, any limit point $(\w,\cB)$ is such that for all $p\geq 2$ 
there exists a constant $C=C(p)$ such that 
for all $\phi\in\cS(\R^2)$
\begin{equ}[e:UpperB]
\mathbf{E}\Big[\Big|\cB_t(\phi)\Big|^p\Big]^\frac1p\leq C (t\vee t^\frac12)\|\phi\|_{\dot H^2(\R^2)}\,,
\end{equ} 
while, for $\theta\in(0,1)$, for all $p\geq 2$ and $\phi\in\cS(\R^2)$ 
\begin{equ}[e:LimNtheta]
\lim_{N\to\infty} \mathbf{E}\Big[\sup_{s\leq t}\Big|\cB^N_s(\phi)\Big|^p\Big]^\frac1p=0\,.
\end{equ}
%
%
%
% and 
%there exists a constant $C=C(p)$ such that for all $N\in\N$ we have 
%\begin{equ}[e:UBtheta]
%\mathbf{E}\Big[\sup_{s\leq t}\Big|\cB^N_s(\phi)\Big|^p\Big]^\frac1p\leq C N^{2\theta-2} (t\vee t^\frac12)\|\phi\|_{\dot H^2(\R^2)}\,.
%\end{equ}
%Further, for $\theta=1$, any limit point $(\w,\cB)$ is such that for all $p\geq 2$ 
%there exists a constant $C=C(p)$ such that 
%for all $\phi\in\cS(\R^2)$
%\begin{equ}[e:UpperB]
%\mathbf{E}\Big[\Big|\cB_t(\phi)\Big|^p\Big]^\frac1p\leq C (t\vee t^\frac12)\|\phi\|_{\dot H^2(\R^2)}\,.
%\end{equ} 
\end{theorem}

\begin{remark}\label{rem:UB}
For $\theta=1$, the previous theorem proves both the tightness of the sequence $\{(\wN, \cB^N)\}_N$ 
stated in Theorem~\ref{thm:Main} and 
the upper bound in~\eqref{e:UpperLowerB}. The latter can be directly verified by 
considering~\eqref{e:UpperB} with $p=2$ and applying 
the Laplace transform at both sides. 
\end{remark}

\begin{proof}
The proof follows the same steps and computations performed in Section~\ref{s:invmeas} 
for Lemma~\ref{l:tightvsnsNT}. 
More precisely, the statements of Lemma~\ref{l:itotrick} (the It\^o trick), Proposition~\ref{p:EnergyEstimates} 
and Lemma~\ref{l:CaosNnlin} 
hold {\it mutatis mutandis} in the non-periodic case - it suffices to remove the superscripts $M$, 
replace every instance of $\T^2_M$ with $\R^2$ and  
substitute the weighted Riemann-sums with integrals. 
Hence, we deduce that for any $\phi\in\cS(\R^2)$ and any $p\geq 2$
\begin{equ}[e:UBtheta]
\mathbf{E}\Big[\sup_{s\leq t}\Big|\cB^N_s(\phi)\Big|^p\Big]^\frac1p\lesssim N^{2\theta-2} (t\vee t^\frac12)\|\phi\|_{\dot H^2(\R^2)}\,.
\end{equ}
which implies tightness for $\cB^N$ for $\theta\in(0,1]$ by Mitoma's and Kolmogorov's criteria, 
and~\eqref{e:LimNtheta} for $\theta\in(0,1)$ and~\eqref{e:UpperB} for $\theta=1$. 
Moreover, arguing as in the proof of Lemma~\ref{l:tightvsnsNT}, one sees that~\eqref{e:SNStightM} 
holds for $\wN$. By invoking once more Mitoma's and Kolmogorov's criteria we
conclude that tightness holds also for $\wN$. 
\end{proof}

\subsection{Lower bound on the nonlinearity for $\theta=1$}\label{s:lowerbound}

As shown in Theorem~\ref{t:tightvsnsN}, the choice of the coupling constant $\l$ in~\eqref{e:CouplingConstant} 
ensures tightness of the sequence $\{\w^N\}_N$ of stationary solutions to~\eqref{e:SNSweak} on $\R^2$ 
and, for $\theta=1$, provides an upper bound on the integral in time of the non-linearity. 
In the proposition below, we determine a matching (up to constants) lower bound on its Laplace transform 
thanks to which the proof of Theorem~\ref{thm:Main} is complete.

\begin{proposition}\label{p:lowerboundsns}
In the same setting as Theorem~\ref{t:tightvsnsN}, let $\theta=1$ and $\cB$ be any limit point of the sequence 
$\cB^N$ in~\eqref{e:IntNonlin}. Then, there exists a constant $C>0$ such that 
for all $\kappa>0$ and $\phi\in\cS(\R^2)$ the lower bound in~\eqref{e:UpperLowerB} holds. 
\end{proposition}
\begin{proof}
For $N\in\N$, let $\cB^N$ be defined according to~\eqref{e:IntNonlin}, 
By \cite[Lemma 5.1]{Cannizzaro.etal2D2021}, for $N\in\N$ we have 
 \begin{equ}[e:laplaceqvar]
    \int_0^{\infty} e^{-\la t} \Ex \Big[ \Big| \cB^N_t(\phi) \Big|^2 \Big] \dd t = \frac{2}{\la^2} \E\Big[\nnlin [\mu] (\phi) (\la - \mathcal{L}^N )^{-1} \nnlin [\mu] (\phi) \Big] \,.
\end{equ}
Thanks to~\cite[Lemma 5.2]{Cannizzaro.etal2D2021} and the isometry $I$ introduced in Section~\ref{sec:Mall}, 
the right hand side above equals
\begin{equs}[e:variationalvsnsN]
&\frac{2}{\la^2 } \sup_{G\in L^2(\mu)} \left\{ 2\E[\lambda_{N,1} \nnlin [\mu] (\phi) G] - \E[G(\la-\cL_0)G ] - \E[\ngena G (\la - \cL_0)^{-1}\ngena G]  \right\}\\
&=\frac{2}{\la^2 } \sup_{g\in \fock} \left\{ 2\qvar{\lambda_{N,1} \fn^N_\phi, g}_{\fock} - \qvar{g,(\la-\cL_0)g }_{\fock} - \qvar{\ngena g, (\la - \cL_0)^{-1}\ngena g}_{\fock}  \right\}
\end{equs}
where $\fn^N_\phi$ is such that $\nnlin [\mu] (\phi)=I_2(\fn^N_\phi)$ and its Fourier transform is 
given by the right hand side of~\eqref{e:KernelNnlin} (for $\ell,m\in\R^2$). 
We can further lower bound~\eqref{e:variationalvsnsN} by restricting to $g$ to $\fock_2$ 
for which, by orthogonality of different chaoses and the properties of $\ngenap$ and 
$\ngenam$ determined in Proposition~\ref{p:SNSstraightoperators}, we have
\begin{equ}
\qvar{\ngena g, (\la - \cL_0)^{-1}\ngena g}_{\fock_2}=\qvar{\ngenap g, (\la - \cL_0)^{-1}\ngenap g}_{\fock_3}+\qvar{\ngenam g, (\la - \cL_0)^{-1}\ngenam g}_{\fock_1}\,.
\end{equ}
Summarising, the left hand side of~\eqref{e:laplaceqvar} is lower bounded by 
\begin{equs}[e:LB2]
\frac{2}{\la^2 } \sup_{g\in \fock_2}\Big\{& 2\qvar{\lambda_{N,1} \fn^N_\phi, g}_{\fock_2} - \qvar{g,(\la-\cL_0)g }_{\fock_2}  \\
&-\qvar{ g, -\ngenam(\la - \cL_0)^{-1}\ngenap g}_{\fock_2}-\qvar{ g, -\ngenap(\la - \cL_0)^{-1}\ngenam g}_{\fock_2}  \Big\}
\end{equs}
where we further exploited that the adjoint of $\ngenap$ is $-\ngenam$ and vice versa.

The operators $-\ngenam(\la - \cL_0)^{-1}\ngenap$ and $-\ngenap(\la - \cL_0)^{-1}\ngenam$, even though 
explicit, are difficult to handle since they are not diagonal in Fourier space, meaning that their 
Fourier transform cannot be expressed in terms of an explicit multiplier. 
Nevertheless, the following lemma, whose proof we postpone to the end of the section, 
ensures that they can be bounded by one.

\begin{lemma}\label{l:OpBounds}
There exists a constant $C>0$ independent of $N$ such that for any $g\in\fock_2$, 
the following bound hold 
\begin{equ}[e:MainOpBound]
\qvar{ g, -\ngenam(\la - \cL_0)^{-1} \ngenap g }_{\fock_2}\vee \qvar{g, -\ngenap(\la - \cL_0)^{-1} \ngenam g }_{\fock_2}\leq C \qvar{(-\cL_0)g , g}_{\fock_2}\, .
\end{equ}
\end{lemma}

Assuming the previous lemma holds, there exists a constant $c>1$ independent of $n$ 
such that~\eqref{e:LB2} is bounded below by 
\begin{equs}[e:LB3]
\frac{2}{\la^2 } \sup_{g\in \fock_2}\Big\{& 2\qvar{\lambda_{N,1} \fn^N_\phi, g}_{\fock_2} - \qvar{g,(\la-c\cL_0)g }_{\fock_2}  \Big\}\\
&=\frac{2}{\la^2 } \sup_{g\in \fock_2}\Big\{ \qvar{\lambda_{N,1} \fn^N_\phi, g}_{\fock_2} + \qvar{\lambda_{N,1} \fn^N_\phi - (\la-c\cL_0)g, g}_{\fock_2}  \Big\}\,.
\end{equs}
Now, in order to prove~\eqref{e:UpperLowerB}, it suffices to exhibit {\it one} $g$ for which the lower bound holds, 
and we choose it in such a way that the second scalar product in the supremum is $0$, 
i.e. we pick $g=\mathfrak{g}$, the latter being the unique solution to
\begin{equ}[e:NewPoissonLB]
\lambda_{N,1} \fn^N_\phi - (\la-c\cL_0)\mathfrak{g}=0\,.%\quad\Longleftrightarrow\quad \mathfrak{g}=\l(\la-\cL_0-\cS_1-\cS_2)^{-1} \fn^N_\phi\,.
\end{equ}
Notice that, by~\eqref{e:gensf}, $\mathfrak{g}$ has an explicit Fourier transform which is given by 
\begin{equ}
\hat{\mathfrak{g}}(k_{1:2})= \lambda_{N,1} \frac{\hat \fn_\phi(k_{1:2})}{\la +\frac{c}2|k_{1:2}|^2}\,.
\end{equ}
Plugging  $\mathfrak{g}$ into~\eqref{e:LB2} we obtain a lower bound of the type 
\begin{equs}[e:LB5]
\frac{2}{\la^2} &\qvar{\l \fn^N_\phi, \mathfrak{g}}_{\fock_2} =\frac{2\lambda_{N,1}^2}{\la^2} \int_{\R^4} |k_1|^2|k_2|^2\frac{|\hat \fn_\phi(k_{1:2})|^2}{\la +\frac{c}2|k_{1:2}|^2}\dd k_{1:2}\\
&=\frac{2}{\la^2} \int_{\R^2}\dd k |\phi_k|^2\Big(\lambda_{N,1}^2 \int_{\R^2}\dd k_{2} |k-k_2|^2|k_2|^2\frac{|\mathcal{K}_{k-k_2,k_2}^{N}|^2}{\la +\frac{c}2(|k-k_2|^2+|k_{2}|^2)}\Big)
\end{equs}
which is fully explicit and we are left to consider the inner integral. 
To do so, recall the definition of $\mathcal{K}^N$ in~\eqref{e:nonlintorusCoeff}. 
We restrict the integral over $k_2$ to the sector 
\begin{equ}
\CC^N_k\eqdef\{k_2\colon \theta_{k_2}\in \theta_k+(\pi/6, \pi/3)\quad\&\quad N/3\geq |k_2|\geq (2|k|)\vee 2/N\quad\&\quad |k|\leq \sqrt{N}\}
\end{equ}
where, for $j\in\R^2$, $\theta_j$ is the angle between the vectors $j$ and $(1,0)$. 
It is not hard to see that, on $\CC_k$, we have 
\begin{equs}
|\mathcal{K}_{k-k_2,k_2}^{N}|^2&=\frac{1}{2\pi} (\molf_{k-k_2,k_2})^2  \frac{|(k-k_2)^\perp\cdot k|^2|k_2\cdot k|^2}{|k_2|^4|k-k_2|^4}=\frac{1}{2\pi} (\molf_{k-k_2,k_2})^2\frac{|k_2\cdot k^\perp|^2|k_2\cdot k|^2}{|k_2|^4|k-k_2|^4}\\
&=\frac{1}{2\pi} (\molf_{k-k_2,k_2})^2\frac{|k|^4}{|k-k_2|^4} |\cos(\theta-\theta_k)|^2|\cos(\theta-\theta_{k^\perp})|^2\geq c_\rho \frac{|k|^4}{|k_2|^2|k-k_2|^2}
\end{equs}
for a constant $c_\rho$ depending only on $\rho$ but neither on $k$ nor $N$. 
In the last step, we used that by assumption~\eqref{ass:rho} on $\rho$, $|\hat\rho^N|$ is bounded below on 
$[2/N,N/2]$ by a constant independent of $N$ and that on $\CC^N_k$ we have 
\begin{equ}
\tfrac2N\leq |k|,|k_2|,|k-k_2|\leq \tfrac{N}2\,,\qquad\text{and}\qquad\tfrac32|k_2|\geq |k-k_2|\geq \tfrac12 |k_2|\,.
\end{equ}
Hence, the right hand side of~\eqref{e:LB5} is lower bounded, 
modulo a multiplicative constant only depending on $\rho$, by 
\begin{equ}[e:LB6]
\frac{2}{\la^2} \int_{2/N\leq|k|\leq\sqrt N}\dd k |k|^4 |\phi_k|^2\Big(\lambda_{N,1}^2 \int_{\CC^N_k} \frac{\dd k_{2}}{\la +|k_{2}|^2}\Big)\,.
\end{equ}
It remains to treat the quantity in parenthesis, for which we pass to polar coordinates and obtain
\begin{equ}
\lambda_{N,1}^2 \int_{\CC^N_k} \frac{\dd k_{2}}{\la +|k_{2}|^2}\geq\lambda_{N,1}^2 \int_{2\sqrt{N}}^{N/3}\frac{\rho\dd \rho}{\kappa +\rho^2}=\frac{\lambda}{2\log N}\log\Big(\frac{\kappa+N^2/9}{\kappa+4N}\Big)\gtrsim1\,.
\end{equ}
In conclusion, we have shown that for $N$ large enough
\begin{equ}[e:LowerBN]
\int_0^{\infty} e^{-\la t} \Ex \Big[ \Big| \cB^N_t(\phi) \Big|^2 \Big] \dd t \gtrsim \frac{1}{\la^2} \int_{2/N\leq|k|\leq\sqrt N}\dd k |k|^4 |\phi_k|^2\,,
\end{equ}
and it remains to pass to the limit as $N\to\infty$. Now, thanks to~\eqref{e:UpperB} and tightness of $\cB^N$, 
we can apply dominated convergence to the left hand side, while the integral at right hand side clearly converges to $\|\phi\|_{\dot H^2(\R^2)}^2$, so that the proof is completed. 
\end{proof}

\begin{proof}[of Lemma~\ref{l:OpBounds}]
We will exploit the Fourier representation of the operators $\ngenap$ and $\ngenam$ in 
Proposition~\ref{p:SNSstraightoperators}, which though still need to be symmetrised.
Let $\mathfrak{a}_+^N$ be the operator defined by the right hand side of~\eqref{e:genapf} and $S_3$ the 
set of permutations of $\{1,2,3\}$. 
Then, 
\begin{equs}[e:offdiags]
    	\qvar{ g, \ngenam &(\la - \ngensy)^{-1} \ngenap g }_{\fock_2} =\qvar{\ngenap g, (\la - \ngensy)^{-1} \ngenap g }_{\fock_3} \\ 
      &= \sum_{s,\bar{s} \in S_3} \int \frac{|k_1|^2 |k_2|^2 |k_3|^2 }{\la + \half |k_{1:3}|^2} \mathcal{F}(\mathfrak{a}_+^N g) (k_{s(1):s(3)} ) \mathcal{F}(\mathfrak{a}_+^N g) (k_{\bar s(1): \bar s (3)}) \dd k_{1:3} \\
      &\lesssim \int \frac{|k_1|^2 |k_2|^2 |k_3|^2 }{\la + \half |k_{1:3}|^2} \mathcal{F}(\mathfrak{a}_+^N g) (k_{1:3} )^2 \dd k_{1:3}
\end{equs}
where in the last step we simply applied Cauchy-Schwarz inequality. 
Now, we bound $|\mathcal{K}_{k_1,k_2}^{N}|\leq \molf_{k_2} |k_1+k_2|^2/(|k_1||k_2|)$ so that the right 
hand side above can be controlled via 
\begin{equs}[e:A+Bound]
\lambda_{N,1}^2 \int_{\R^6}  &\hat{g} (k_1 +k_2 , k_3 )^2   \hat\rho_{k_2}\frac{|k_3|^2 |k_1 + k_2|^4 }{\la + \half |k_{1:3}|^2 } \dd k_{1:3}\\
&\lesssim \int_{\R^4}  \dd k_{1:2}\Big(\prod_{i=1}^2|k_i|^2\Big)|k_1|^2|\hat{g} (k_1 , k_2 )|^2 \Big(\lambda_{N,1}^2\int_{\R^2} \frac{\hat\rho_{j} \dd j }{\la + |j|^2 } \Big)\\
&\lesssim \int_{\R^4}  \dd k_{1:2}\Big(\prod_{i=1}^2|k_i|^2\Big)|k_1|^2|\hat{g} (k_1 , k_2 )|^2\\
&=\frac12 \int_{\R^4}  \dd k_{1:2}\Big(\prod_{i=1}^2|k_i|^2\Big)(|k_1|^2+|k_2|^2)|\hat{g} (k_1 , k_2 )|^2=\qvar{(-\CL_0) g, g}_{\fock_2}
\end{equs}
where the second step follows by the fact that $\hat\rho_j\leq \1_{|j|\leq N}$ 
and the definition of $\lambda_{N,1}$ in~\eqref{e:CouplingConstant}, while the last by the symmetrisation of the integral. 

We now turn to the other term, which is 
\begin{equs}
&\qvar{ g, \ngenap (\la - \ngensy{0})^{-1} \ngenam g }_{\fock_2} \\
&=\qvar{\ngenam g, (\la - \ngensy)^{-1} \ngenam g }_{\fock_1}\lesssim\lambda_{N,1}^2  \int \frac{|k|^2}{\la + \half |k|^2 }  \F(\ngenam g)(k)^2  \dd k\\
&=\lambda_{N,1}^2\int_{\R^2} \dd k\frac{|k|^2}{\la + \half |k|^2 } \Big(\int_{\R^2} \molf_{p,k-p} \frac{(k^{\perp} \cdot p)(k\cdot (k-p))}{|k|^2} \hat g(p,k-p)\dd p\Big)^2\\
&\lesssim \lambda_{N,1}^2 \int_{\R^2} \dd k \Big(\int_{\R^2} \molf_{p} |p||k-p|  \hat g(p,k-p)\dd p\Big)^2\,.
\end{equs}
We now multiply and divide the integrand by $|p|$ and apply Cauchy-Schwarz, so that 
we obtain an upper bound of the form
\begin{equ}
\Big(\int_{\R^4} \dd k_{1:2}\Big(\prod_{i=1}^2|k_i|^2\Big)|k_1|^2\hat{g} (k_1 , k_2 )|^2\Big)\Big(\lambda_{N,1}^2\int_{\R^2}(\molf_{p})^2 \frac{\dd p}{|p|^2}\Big)
\end{equ}
from which~\eqref{e:MainOpBound} follows arguing as in~\eqref{e:A+Bound}. 
\end{proof}

\subsection{Triviality of the fractional vorticity equation for $\theta<1$}\label{s:triv}

In this last section, we complete the proof of Theorem~\ref{thm:theta<1} and show that the rescaled solution of the 
regularised fractional vorticity equation for $\theta\in(0,1)$ converges to the fractional stochastic heat equation 
obtained by simply setting the coupling constant $\lambda$ in~\eqref{e:vsnsN} to $0$. 

For the proof, recall that $\w$ is a stationary (analytically) weak solution of~\eqref{e:fracSHE} if 
for all $\phi\in\cS(\R^2)$, $\w$ satisfies
\begin{equ}
\w_t(\phi)=\mu(\phi)+\int_0^t \w_s(-(-\Delta)^\theta\phi)\dd s+\int_0^t \xi(\dd s, (-\Delta)^{\frac{1+\theta}{2}}\phi)
\end{equ}
where $\mu$ is the Gaussian process whose covariance is given by~\eqref{e:mu}. 
It is not hard to see that $\w$ admits a unique stationary weak solution. 
This is the only tool we need for the proof, which is then a simple corollary of Theorem~\ref{t:tightvsnsN}. 

\begin{proof}[of Theorem~\ref{thm:theta<1}]
For $N\in\N$, let $\wN$ be a stationary weak solution to~\eqref{e:vsnsN}, i.e. for all $\phi\in\cS(\R^2)$ 
$\wN$ satisfies
\begin{equ}
\omega_t^N(\phi)-\mu(\phi)=\half \int_0^t \omega_s^N(-(-\Delta)^\theta\phi)\dd s +\cB^N_t(\phi) -
	\int_0^t \xi(\dd s, (-\Delta)^{\frac{1+\theta}{2}}\phi)\,,
\end{equ}
where $\cB^N$ is defined according to~\eqref{e:IntNonlin}. 
By Theorem~\ref{t:tightvsnsN}, the sequence $(\wN,\cB^N)$ is tight in the space $C([0,T],\cS'(\R^2))$ 
and, thanks to~\eqref{e:LimNtheta}, $\cB^N\to 0$ as $N\to\infty$. Hence, it immediate 
to verify that every limit point of $\wN$ is a weak stationary solution 
of~\eqref{e:fracSHE}. Since the latter is unique, the result follows at once. 
\end{proof}

\bibliography{library}
\bibliographystyle{Martin}
\end{document}